\documentclass{amsart}

\usepackage{amssymb}
\usepackage{mathtools}
\usepackage{leftidx}
\usepackage{stmaryrd}
\usepackage{url}
\usepackage[OT2,T1]{fontenc}
\usepackage{graphicx}
\usepackage[cmtip,all]{xy}

\DeclareSymbolFont{cyrletters}{OT2}{wncyr}{m}{n}
\DeclareMathSymbol{\Sh}{\mathalpha}{cyrletters}{"58}

\DeclareMathOperator{\SL}{SL}
\DeclareMathOperator{\gal}{Gal}

\DeclareMathOperator{\fr}{Fr}
\DeclareMathOperator{\aut}{Aut}
\DeclareMathOperator{\tr}{Tr}
\DeclareMathOperator{\en}{End}
\DeclareMathOperator{\res}{Res}
\DeclareMathOperator{\id}{id}
\DeclareMathOperator{\ord}{ord}
\DeclareMathOperator{\Hom}{Hom}
\DeclareMathOperator{\gr}{Gr}

\newtheorem{theorem}{Theorem}[section]
\newtheorem{lemma}[theorem]{Lemma}
\newtheorem{proposition}[theorem]{Proposition}
\newtheorem{conjecture}[theorem]{Conjecture}
\newtheorem*{theorem*}{Theorem}

\theoremstyle{definition}
\newtheorem{definition}[theorem]{Definition}
\newtheorem{example}[theorem]{Example}

\theoremstyle{remark}
\newtheorem{remark}[theorem]{Remark}

\numberwithin{equation}{section}

\newcommand{\N}{\mathbb{N}}
\newcommand{\Z}{\mathbb{Z}}
\newcommand{\Q}{\mathbb{Q}}
\newcommand{\R}{\mathbb{R}}
\newcommand{\C}{\mathbb{C}}
\newcommand{\F}{\mathbb{F}}
\newcommand{\OQ}{\overline{\Q}}
\newcommand{\GQ}{G_{\Q}}
\newcommand{\p}{\mathfrak{p}}

\newcommand{\s}{{}^\sigma\!}
\newcommand{\n}{{}^\nu\!}
\newcommand{\barr}{\left(\begin{array}}
\newcommand{\narr}{\end{array}\right)}
\newcommand{\sub}{\subseteq}

\begin{document}

\title{On $L$-functions of quadratic $\Q$-curves}

\author[P.\thinspace J. Bruin]{Peter Bruin}
\address[1]{Mathematisch Instituut\\
Universiteit Leiden\\
Niels Bohrweg 1\\
2333 CA Leiden, Netherlands\\
}
\email{P.J.Bruin@math.leidenuniv.nl}
\thanks{The first author was partially supported by the Swiss National Science Foundation through grants 124737 and 137928, and by the Netherlands Organisation for Scientific Research (NWO) through Veni grant 639.031.346. The second author was partially supported by Swiss National Science Foundation grant 168459.}

\author[A. Ferraguti]{Andrea Ferraguti}
\address[2]{Institut f\"{u}r Mathematik\\
Universit\"{a}t Z\"{u}rich\\
Winterthurerstrasse 190\\
8057 Z\"{u}rich, Switzerland\\
}
\address[3]{University of Cambridge\\
DPMMS\\
Centre for Mathematical Sciences\\
Wilberforce Road, Cambridge, CB3 0WB, UK\\}
\email{andrea.ferraguti@math.uzh.ch}


\begin{abstract}
Let $K$ be a quadratic number field and let $E$ be a $\Q$-curve without CM completely defined over $K$ and not isogenous to an elliptic curve over $\Q$. In this setting, it is known that there exists a weight $2$ newform of suitable level and character, such that $L(E,s)=L(f,s)L(\s f,s)$, where $\s f$ is the unique Galois conjugate of $f$. In this paper, we first describe an algorithm to compute the level, the character and the Fourier coefficients of $f$. Next, we show that given an invariant differential $\omega_E$ on $E$, there exists a positive integer $Q=Q(E,\omega_E)$ such that $L(E,1)/P(E/K)\cdot Q$ is an integer, where $P(E/K)$ is the period of $E$. Assuming a generalization of Manin's conjecture, the integer $Q$ is made effective. As an application, we verify the weak BSD conjecture for some curves of rank two, we compute the $L$-ratio of a curve of rank zero and we produce relevant examples of newforms of large level.
\end{abstract}

\maketitle
	\section{Introduction}
	Let $E$ be an elliptic curve defined over a number field $K$ and let $L(E,s)$ be its $L$-function. This is a holomorphic function defined on the half plane $\{s\in \C\colon \Re(s)>3/2\}$. For a certain class of elliptic curves, and conjecturally for every elliptic curve, $L(E,s)$ has an analytic continuation to $\C$. In these cases, it is of deep interest to know the order of vanishing of $L(E,s)$ in $s=1$, which is called the \emph{analytic rank} of $E$. The main reason of interest is certainly the Birch and Swinnerton-Dyer conjecture, which, in its weak form, asserts that the analytic rank and the algebraic rank of $E$ coincide. The conjecture is known to be true over $\Q$ when the analytic rank is at most $1$ (see \cite{groza}, \cite{kol} and \cite{mur} or the survey in \cite{gro}), but very little is known in the general case; moreover it is extremely difficult even to verify the BSD conjecture for a given $E$.

	Suppose now that $K$ is a quadratic number field of discriminant $\Delta_K$ and that $E$ is a $\Q$-curve completely defined over $K$ (i.e.\ such that $E$ is $K$-isogenous to its Galois conjugate). This is a sufficient condition to ensure that $L(E,s)$ can be analytically continued to $\C$. In the present paper we address the following problem: how can we decide whether $L(E,1)$ vanishes or not? Computations with modular symbols can in principle answer the question, but they are inefficient when the conductor of $E$ is large. Alternatively, one can compute $L(E,1)$ to any given precision; however, it is not a priori clear how to decide whether $L(E,1)$ is exactly $0$ or just a very small non-zero number. The same type of problem arises when $L(E,1)\neq 0$: let $P(E/K)$ be the period of $E$ (cf.\ subsection \ref{BSD}). This coincides with the product of the Tamagawa numbers of $E$ with $2^s\cdot\Omega_E/\sqrt{|\Delta_K|}$. Here $s=0$ if $K$ is complex and $s\in\{0,1,2\}$ if $K$ is real (depending on the $2$-torsion of $E$), while $\Omega_E$ is the product of the real periods of $E$ when $K$ is real and the covolume of the period lattice when $K$ is imaginary (see section \ref{mainsec} for the precise definition); this can be computed efficiently (see for example \cite{creto}). Suppose that we can compute the $L$-ratio $L(E,1)\cdot \sqrt{|\Delta_K|}/\Omega_E$ to any given precision, finding a value which is very close to a rational number $t$. How can we \emph{prove} that the $L$-ratio is exactly $t$?

	Our starting point, in section \ref{curvesoverQ}, will be elliptic curves over $\Q$. As stated by the celebrated modularity theorem (see \cite{wil}, \cite{tawil} and \cite{bcdt} for the original proof, or \cite{dis} for a survey), these curves admit a non-trivial map $X_0(N)\to E$, called a \emph{modular parametrization}, where $N$ is the conductor of $E$ and $X_0(N)$ is the compact modular curve for $\Gamma_0(N)$.  A consequence of this fact is that if $\pi\colon X_0(N)\to E$ is a modular parametrization and $\omega_E$ is a N\'{e}ron differential on $E$, then $\pi^*(\omega_E)=c\cdot f$, where $c=c(E,\pi)$ is a non-zero integer (defined up to sign) called the \emph{Manin constant} and $f\in S_2(\Gamma_0(N))$ is a newform. The $L$-function attached to $f$ coincides with the $L$-function of $E$ and one can see that $L(f,1)=-2\pi i\int_{0}^{i\infty}f(t)dt$ using the formula for the analytic continuation of $L(f,s)$. Now a theorem of Manin and Drinfel'd (see \cite{drin} and \cite{man1}) shows that $\pi(0)-\pi(i\infty)$ has finite order in $E(\Q)$ and this allows us to relate $L(f,1)$ to the real period $\Omega_E$ of $E$ and $c$. In general it is a very hard problem to compute $c$, but assuming Manin's conjecture it is possible to find an explicit multiple of $c$ in terms of the $\Q$-isogeny class of $E$. Eventually, this will permit us to find an effective positive integer $Q=Q(E,\omega_E)$ such that $L(E,1)\cdot Q/\Omega_E$ is a non-zero integer whenever $L(E,1)\neq 0$. This gives us a very efficient method to decide whether $L(E,1)=0$. In fact this happens if and only if computing $L(E,1)$ up to a sufficient precision it results that $|L(E,1)|<\Omega_E/Q$. Moreover, the same result allows us to compute the rational number $L(E,1)/\Omega_E$ whenever this is different from $0$.

	Over $\OQ$ the situation is more complicated. Although there are different uses for the word ``modular'' in the literature, throughout this paper we call an elliptic curve $E$ over $\OQ$ \emph{modular} if there exists some $M \in \N$ such that there is a non-constant map $X_1(M)\to	E$. Ribet showed in \cite{rib1} that Serre's conjecture on $\bmod$ $p$ Galois representations (which was later proved in \cite{kwin} and \cite{kwin2}) implies that modular elliptic curves without complex multiplication (CM) are exactly $\Q$-curves without CM, that is, elliptic curves without CM which are isogenous to all of their Galois conjugates. In contrast to the case of elliptic curves over $\Q$, the $L$-function	of a $\Q$-curve does not necessarily coincide with the $L$-function attached to a
	newform. However, within the class of $\Q$-curves it is possible to define the subclass of the so-called \emph{strongly modular curves}, which are characterized by the property that $L(E, s)$ is a product of $L$-functions of newforms (see \cite{guique}).

	After reviewing in sections \ref{strmod} and \ref{buiblo} some basic constructions associated to $\Q$-curves taken from \cite{gola}, \cite{guique}, \cite{que} and \cite{rib1}, in section \ref{quacur} we will focus our attention on quadratic $\Q$-curves completely defined over a quadratic number field $K$. All such curves are strongly modular and Ribet's theorem ensures the existence of a newform $f\in S_2(\Gamma_1(N))$ such that $L(E,s)=L(f,s)L(\s f,s)$, where $\s f$ is the unique Galois conjugate of $f$. Section \ref{newform} is dedicated to explaining how one can compute the Fourier coefficients of $f$ given $E$ and an isogeny from $E$ to its Galois conjugate $\n E$.

	The existence of the newform $f$ follows from the fact that the Weil restriction of scalars of $E$, which is an abelian surface over $\Q$, is $\Q$-isogenous to the abelian variety $A_f$ attached to $f$ by Shimura (see \cite{shi}). This is the key fact that will allow us to generalize the ``geometric'' argument used for elliptic curves over $\Q$ to the case of quadratic $\Q$-curves. Section \ref{mainsect} contains the core of our argument. We will show there how to choose an appropriate parametrization starting from the data of $E$, an invariant differential $\omega_E$ and an isogeny $\mu\colon E\to \n E$, and how to apply the Manin-Drinfel'd theorem to this setting in order to again relate $L(E,1)$ to the period of $E$ and the discriminant of $K$.

	One fundamental difference with the case of elliptic curves over $\Q$ is the fact that there is no direct way to uniquely define a Manin constant. In fact if $\mathcal E$ is a N\'{e}ron model for $E$ over the ring of integers $\mathcal O_K$ of $K$, then $H^0(\mathcal E,\Omega^1_{\mathcal E/\mathcal O_K})$ is a locally free $\mathcal O_K$-module of rank $1$, but it is not necessarily free. However, the pullback of $\omega_E$ under our modular parametrization coincides with $\gamma\cdot h$, for certain $\gamma\in K^*$ and $h\in \langle f,\s f\rangle_{\C}$. In section \ref{manid} we will recall, following \cite{gola}, the definition of the so-called \emph{Manin ideal}, an invariant attached to a modular parametrization $X_1(N)\to E$. Assuming a generalization of Manin's conjecture we will be able to use properties of the Manin ideal to find an explicit rational number whose quotient by $N_{K/\Q}(\gamma)$ is an integer; this, together with a small computation performed in section \ref{finbnd}, will finally allow us to compute an effective positive integer $Q=Q(E,\omega_E)$ such that $\displaystyle L(E,1)\cdot \frac{\sqrt{|\Delta_K|}}{\Omega_E}\cdot Q$ is an integer.

	The main result, stated in section \ref{mainsec}, can be summarized as follows: 
		\begin{theorem*}
			Let $K$ be a quadratic number field with discriminant $\Delta_K$ and let $E$ be a quadratic $\Q$-curve completely defined over $K$. For every finite place $v$ of $K$, let $c_v$ be the Tamagawa number of $E$ at $v$. Let $\displaystyle P(E/K)=\prod_vc_v\cdot 2^s\cdot \frac{\Omega_E}{\sqrt{|\Delta_K|}}$ be the period of $E$. Suppose that $L(E,1)\neq 0$. Then:
			$$L(E,1)\cdot \frac{\sqrt{|\Delta_K|}}{\Omega_E}\in \Q^*.$$
			Moreover, assuming the generalized Manin conjecture, if we fix an invariant differential $\omega_E$ then there exists an effective positive integer $Q=Q(E,\omega_E)$ such that $\displaystyle L(E,1)\cdot \frac{\sqrt{|\Delta_K|}}{\Omega_E}\cdot Q$ is an integer.
	\end{theorem*}
	In section \ref{exs} we will produce, starting from quadratic $\Q$-curves of algebraic rank two, relevant examples of newforms of large level which cannot feasibly be computed using modular symbols. Finally, we will show how to use our result to prove that the analytic rank of these curves is exactly two and how the theorem can be used to compute the $L$-ratio when the analytic rank of the curve is zero.

        \subsection*{Acknowledgments}
        We would like to thank John Cremona for several helpful
        comments.
	\section*{Notation and conventions}
	For algebraic varieties $A,B$ defined over a field $K$, when we talk about maps $\varphi\colon A\to B$ we always mean, unless specified otherwise, that $\varphi$ is also defined over $K$. If in addition $A$ and $B$ are abelian varieties, $\Hom(A,B)$ is the set of isogenies $A\to B$ defined over $K$, while if $F$ is an extension of $K$, the set of isogenies defined over $F$ is denoted by $\Hom_F(A,B)$. In particular, when $A=B$, the $\Q$-algebra of the $F$-endomorphisms of $A$ is denoted by $\en_F^0(A)$. The dual of an isogeny $\varphi$ is denoted by $\widehat{\varphi}$. The base-change of $A$ to $F$ is denoted by $A_F$.

	The word ``newform'' means ``normalized newform'', so if $\displaystyle \sum_{n=1}^{+\infty} a_nq^n$ is the Fourier expansion of a newform then $a_1=1$.

        For a number field $K$, the absolute Galois group of $K$ is
        denoted by $G_K$. All our number fields are contained in a
        fixed algebraic closure $\overline{\Q}$ of $\Q$.
        
	If $F/K$ is a Galois extension of fields and $A$ is a $\gal(F/K)$-module, we denote by $H^i(F/K,A)$ the $i$-th cohomology group of $A$ with coefficients in $\gal(F/K)$. Group actions will be denoted on the upper left corner. For example if $\sigma \in \gal(F/K)$ and $a\in A$ then $\s a$ is the element $a$ acted on by $\sigma$.

\section{Elliptic curves over $\Q$}\label{curvesoverQ}
		Let $E$ be an elliptic curve over $\Q$ with conductor $N$. In this section, we are going to recall how it is possible to relate the special value $L(E,1)$ to the period of $E$, exploiting the modularity theorem. For a reference on this topic, see for example \cite{darmon}. This will serve us as a model for the more general situation that we will study in the subsequent sections of the paper.

		A \emph{modular parametrization} is a non-constant map of algebraic curves $\pi \colon X_0(M) \to E$ for some $M\in \N$. By the modularity theorem (see \cite{bcdt},\cite{wil} and \cite{tawil}), such a map always exists with $M=N$. If $\omega_E$ is a N\'{e}ron differential on $E$, the multiplicity one principle shows that the pullback $\pi^*(\omega_E)$ is a multiple of a newform $f\in S_2(\Gamma_0(N))$ by a constant $c\in \Q^*$, called the \emph{Manin constant}. Note that choosing $\omega_E$ is equivalent to choosing the sign of $c$.
		\begin{theorem}[Edixhoven, \cite{edi}]\label{manconisint}
				The Manin constant is an integer.
		\end{theorem}
		Let now $E'$ be another elliptic curve over $\Q$ of conductor $N$, and let $\pi\colon X_0(N)\to E$ and $\pi'\colon X_0(N)\to E'$ be two modular parametrizations. We say that $\pi'$ 
		\emph{dominates} $\pi$ if there exists a $\Q$-isogeny $\psi\colon E'\to E$ such that $\psi\circ \pi'=\pi$. This relation defines a partial ordering on the set of isomorphism classes of pairs $(\pi',E')$, where $E'$ is an elliptic curve $\Q$-isogenous to $E$ and $\pi'\colon X_0(N)\to E'$ is a modular parametrization. We write $(\pi',E')\geq (\pi,E)$ if $\pi'$ dominates $\pi$. The map $\pi$ is called a \emph{strong modular parametrization} if it is a maximal element with respect to this ordering. It is clear that a strong parametrization is unique up to isomorphism; moreover it can be shown that every modular parametrization factors through a strong one. The Manin conjecture for elliptic curves over $\Q$ can be stated as follows:
			\begin{conjecture}[Manin conjecture]\label{mancon}
				The Manin constant of a strong parametrization is $\pm 1$.
		\end{conjecture}
		Results in this direction are presented in \cite{abul}, \cite{ars}, \cite{edi} and \cite{maz}.
		Throughout this section, let us fix a modular parametrization $\pi\colon X_0(N)\to E$. Let $\displaystyle f(z)=\sum_{n=1}^{+\infty}a_ne^{2\pi iz}\in S_2(\Gamma_0(N))$ be the newform attached to $E$ by the modularity theorem. One of the consequences of the theorem is that the $L$-function of $E$ coincides with the $L$-function of $f$. The formula for the analytic continuation of the $L$-function of $f$ shows us that
  		$$L(f,s)=\frac{(2\pi)^s}{\Gamma(s)}\int_0^{\infty}f(it)t^s\frac{dt}{t},$$
  		and therefore we have
 			\begin{equation}\label{Lvalue}
  				L(E,1)=L(f,1)=-2\pi i\int_0^{i\infty}f(z)dz.
 		\end{equation}
		Since $\pi^*(\omega_E)=c\cdot f$ for some $c\in \Z$, one has that
			\begin{equation}\label{Lvalue2}
				c\cdot L(f,1)=c\cdot \int_{\{0,i\infty\}}\frac{f(q)}{q}dq=\int_{\pi_*(\{0,i\infty\})}\omega_E,
		\end{equation}
		where $\{0,i\infty\}$ denotes the image in $H_1(X_0(N),\R)$ of any path from $0$ to $i\infty$ in the compactified upper half plane $\mathcal H^*$ and $\pi_*$ denotes the 
		induced map $H_1(X_0(N),\R)\to H_1(E,\R)$.
		Note that $\pi$ maps points in $\mathcal H^*$ lying on the imaginary axis to real points of $E$, because complex 
		conjugation on $X_0(N)$ corresponds to reflection with respect to the imaginary axis in $\mathcal H^*$, and $\pi$ commutes with 
		complex conjugation since it is defined over $\Q$. Moreover, the cusps $0$ and $i\infty$ of $\Gamma_0(N)$ are defined over $\Q$ and therefore $\pi(0),\pi(i\infty)\in 
		E(\Q)$, but not necessarily $\pi(0)=\pi(i\infty)$. However, we have the following result.
			\begin{theorem}[Manin--Drinfel'd, \cite{drin} and \cite{man1}]\label{mandri}
				Let $G$ be a congruence subgroup of $\SL_2(\Z)$ and let $X_G$ be the corresponding modular curve. If $\alpha,\beta$ are 
				two cusps for $G$, then the class $\{\alpha,\beta\}\in H_1(X_G,\R)$ belongs to $H_1(X_G,\Q)$.
		\end{theorem}
		As an immediate corollary, $\pi(0)-\pi(i\infty)$ is a torsion point in $E(\Q)$. If $t=|E(\Q)_{\text{tors}}|$ then $t\pi_*(\{0,i\infty\})\in H_1(E,\Z)$ and so
		$t\pi(0)=t\pi(\infty)$. 		
		Since $E$ is defined over $\R$, complex conjugation on $E(\C)$ defines a involution $E(\C)\to E(\C)$ and consequently 
		an involution $\iota\colon H_1(E(\C),\Z)\to H_1(E(\C),\Z)$. Using the uniformization theorem for elliptic curves, it is easy to see (cf. Lemma \ref{reallattice}) that there exists a $\Z$-basis $\{\gamma_1,\gamma_2\}$ of $H_1(E(\C),\Z)$ such that $\iota(\gamma_1)=\gamma_1$. The real period of $E$ is defined as $\Omega_E\coloneqq \int_{\gamma_1}\omega_E$; up to 
		replacing $\gamma_1$ by $-\gamma_1$ we can assume that $\Omega_E>0$. By what we said above, it follows that $t\pi_*(\{0,i\infty\})=M\gamma_1$ for some $M\in \Z$. Putting everything together, we finally get
			\begin{equation}\label{Lvalue3}
    				L(E,1)=\frac{M\cdot \Omega_E}{c\cdot |E(\Q)_{\text{tors}}|}.
		\end{equation}	
		Now we would like to deduce from \eqref{Lvalue3} an effective integer $Q$ such that $\displaystyle\frac{L(E,1)}{\Omega_E}\cdot Q$ is a non-zero integer, under the assumption that $L(E,1)\neq 0$. Since  $|E(\Q)_{\text{tors}}|$ and $\Omega_E$ can easily be computed (see for example \cite[Algorithm 7.4.7]{coh} or \cite{creto}), this would give us an efficient way to decide if $L(E,1)$ vanishes or not, by computing $L(f,1)=L(E,1)$ with a sufficient precision, and to compute $\displaystyle\frac{L(E,1)}{\Omega_E}$ whenever $L(E,1)\neq 0$. In order to do this, we need to find an explicit multiple of $c$. In principle one could assume that $E$ is the strong curve in its isogeny class, so that under conjecture \ref{mancon} one can assume that $c=1$, then compute its period and finally use the absolute bound on $|E(\Q)_{\text{tors}}|$ to get our desired $Q$. In \cite{gold}, the author proves that there is an algorithm which allows us to do that in polynomial time with respect to the conductor of $E$. However, our philosophy is to avoid computing with modular symbols, since this can take a huge amount of time when the conductor of $E$ is large. Therefore we will show how to find a multiple of $c$, which depends only on $E$, assuming conjecture \ref{mancon} and a certain condition on $\pi$ that we will explain below. For the rest of this section, we will assume that $E$ does not have CM. Let $\pi'\colon X_0(N)\to E'$ be a strong modular parametrization which dominates $\pi$. Let $\omega_{E'}$ be a N\'{e}ron differential on $E'$. Let $\psi\colon E'\to E$ be the isogeny such that $\psi\circ\pi'=\pi$. Then $\psi^*(\omega_E)=c\cdot \omega_{E'}$ and since the dual isogeny $\widehat{\psi}$ extends to a map of N\'{e}ron models it is clear that $\widehat{\psi}^*(\omega_{E'})=a\cdot\omega_E$ for some $a\in \Z$. Since $\psi\circ \widehat{\psi}$ coincides with multiplication by $\deg\psi$, we see that $c$ divides $\deg\psi$. Now let $\varphi$ be an element of minimal degree in $\Hom(E',E)$. Since $E$ does not have CM, there exists some non-zero $m\in \mathbb Z$ such that $\psi=m\varphi$. Then the map $\overline{\pi}\coloneqq\varphi\circ\pi'\colon X_0(N)\to E$ is again a modular parametrization of $E$. Clearly such a parametrization has Manin constant equal to $c/m$. The same computations of $L(E,1)$ that led us to \eqref{Lvalue3} can be performed using $\overline{\pi}\colon X_0(N)\to E$, leading us this time to:
			\begin{equation}\label{Lvalue4}
    				L(E,1)=\frac{m\cdot M'\cdot \Omega_E}{c\cdot |E(\Q)_{\text{tors}}|}
		\end{equation}
		for some other $M'\in \Z$. Both \eqref{Lvalue3} and \eqref{Lvalue4} give us an integral multiple of the same rational number:
		$$L(E,1)= \frac{\Omega_E\cdot v}{c\cdot |E(\Q)_{\text{tors}}|} \mbox{ for some } v\in \Z.$$
		This argument shows that for our purpose we can assume that $\psi=\varphi$. Now we can proceed in the following way: first we compute the curves $E_1,\dots,E_n$ in the $\Q$-isogeny class of $E$; then for each $i=1,\ldots,n$ we set $s_i\coloneqq \min\{\deg\varphi\colon \varphi \in \Hom_{\Q}(E_i,E)\}$ and finally we let $\displaystyle s\coloneqq \gcd(s_i\colon i=1,\ldots,n)$. Since, as we said above, $c$ divides $\deg\psi$, then $c$ divides $s$. Therefore we get that
		\begin{equation}\label{stima}
			\frac{L(E,1)}{\Omega_E}=\frac{v}{s\cdot|E(\Q)_{\text{tors}}|} \mbox{ for some } v\in \Z.
		\end{equation}
		Equation \eqref{stima} has two immediate applications. The first one is the following: assume that $L(E,1)\neq 0$ and that we want to compute the $L$-ratio $\displaystyle \frac{L(E,1)}{\Omega_E}$. This is a rational number of which we know a multiple of the denominator, namely $s\cdot|E(\Q)_{\text{tors}}|$. Now recall the following elementary lemma.
			\begin{lemma}\label{Lratio}
				Let $B\in \N_{>1}$. Then for every $x\in \R$ there exists at most one $p/q\in \Q$ with $q$ a positive divisor of $B$ such that $\displaystyle |x-p/q|<\frac{1}{2B}$.
		\end{lemma}
			\begin{proof}
				Let $p/q$ and $r/s$ be two distinct rational numbers such that $q,s$ are positive divisors of $B$. Let $l=\text{lcm}(q,s)$, so that $l\leq B$. Then
				$$\left|\frac{p}{q}-\frac{r}{s}\right|=\frac{|p\cdot (l/q)-r\cdot(l/s)|}{l}\geq \frac{1}{B}.$$
				This shows that inside an open interval of length $1/B$ there is at most one rational number with the denominator dividing $B$, and the claim follows.
		\end{proof}
		Notice that the bound given in the lemma is sharp, since if $\displaystyle x=\frac{3}{2B}$, then $|x-1/B|=|x-2/B|=1/(2B)$.

		By equation \eqref{stima}, the $L$-ratio $\displaystyle \frac{L(E,1)}{\Omega_E}$ is a rational number whose denominator divides $B\coloneqq s\cdot |E(\Q)_{\text{tors}}|$. Suppose that one can numerically compute $\displaystyle \frac{L(E,1)}{\Omega_E}$ within a sufficiently high precision.  Let $x$ be the approximate value found; the exact value of $\displaystyle\frac{L(E,1)}{\Omega_E}$ is by the above lemma the unique rational number of the form $\lfloor x \rfloor+A/B$ where $A\in \Z$ is such that $|A|<B$ and such that $\displaystyle \left|x-\left(\lfloor x \rfloor+\frac{A}{B}\right)\right|<\frac{1}{2B}$. Equivalently, $A$ is the unique integer such that
		$$|B(x-\lfloor x \rfloor)-A|<\frac{1}{2}.$$

		The second application is the following: suppose that we can compute $L(E,1)$, finding $0$ within a given precision. How can we decide whether the value is exactly $0$ or a very small non-zero number? Equation \ref{stima} tells us that if $L(E,1)\neq 0$ then
		\begin{equation}
		|L(E,1)|\geq \frac{\Omega_E}{s\cdot|E(\Q)_{\text{tors}}|}.
		\end{equation}
		Therefore if we find numerically that $\displaystyle L(E,1)<\frac{\Omega_E}{s\cdot|E(\Q)_{\text{tors}}|}$, then we must have $L(E,1)=0$. Note that for this purpose one can substitute $s$ in the equation above by $s'\coloneqq \max_i\{s_i\}$ where the $s_i$'s are defined as above; in fact it is clear that $\deg\psi\leq s'$.

  	\section{Modular and strongly modular elliptic curves over $\OQ$}\label{strmod}
		Our goal is to get an analogue of equation \eqref{stima} for a more general class of elliptic curves. We say that an elliptic curve $E/\OQ$ is \emph{modular} if there exists $N\in \N$ and a non-constant map of algebraic curves $X_1(N)_{\OQ}\to E$. Note that elliptic curves over $\Q$ are modular in this sense, since for every $N$ there is a natural map $X_1(N)\to X_0(N)$.
  			\begin{definition}
			Let $K$ be a Galois extension of~$\Q$ inside~$\OQ$. An elliptic curve $E/K$ is called a \emph{$\Q$-curve} if for every $\sigma \in\gal(K/\Q)$ there exists an isogeny $\mu_{\sigma}\colon \s E\to E$. We say that $E$ is \emph{completely defined} over~$K$ if $E$ is defined over $K$ and all $\OQ$-isogenies between the $\s E$ are defined over $K$.
		\end{definition}
		By the theory of complex multiplication, every elliptic curve with CM is a $\Q$-curve. From now on, we will always assume that our $\Q$-curves do not have CM.

		If $E/\OQ$ is a $\Q$-curve, one can define a $2$-cocycle in the following way: for every $\sigma\in \GQ$ choose an isogeny $\mu_{\sigma}\colon \s E\to E$ so that the system $\{\mu_{\sigma}\}_{\sigma\in \GQ}$ is locally constant. Then let
		$$\xi(E)\colon \GQ\times \GQ\to \Q^*$$
		$$(\sigma,\tau)\mapsto \mu_{\sigma}\s\mu_{\tau}\mu_{\sigma\tau}^{-1}$$
		where we identified $\en(E)\otimes \Q\simeq\Q$. This is a $2$-cocycle whose class depends only on the isogeny class of $E$ and not on the choice of the $\mu_{\sigma}$. If $E$ is completely defined over a number field $L$, one can choose $L$-isogenies $\mu_{\sigma}$ for every $\sigma \in \gal(L/\Q)$ and in the same way obtain a $2$-cocycle $\xi_L(E)$ whose class in $H^2(L/\Q,\Q^*)$ depends only on the $L$-isogeny class of $E$.
		\begin{theorem}[{{Khare--Wintenberger \cite{kwin},\cite{kwin2} and Ribet \cite{rib1}}}]
                \label{Qcurvesaremodular}
				An elliptic curve $E/\OQ$ without CM is modular if and only if it is a $\Q$-curve.
		\end{theorem}
		Ribet proved this theorem assuming Serre's conjecture, which was later proved in \cite{kwin} and \cite{kwin2}. The idea of the proof is essentially as follows. One picks a Galois number field $L$ over which $E$ is completely defined and such that the class of $\xi_L(E)$ in $H^2(L/\Q,\OQ^*)$ is trivial. Such a number field always exists because of a theorem of Tate (see \cite[Theorem 4]{ser}). Then there exists a newform $f\in S_2(\Gamma_1(N),\varepsilon)$ for some $N,\varepsilon$ such that the abelian variety $A_f$ attached to $f$ (for details on the construction of $A_f$, see \cite{shi1}) is one of the factors up to isogeny of the abelian variety $\res_{L/\Q}(E)$. The curve $E$ is then a quotient of $(A_f)_{\OQ}$, and composing the quotient map $(A_f)_{\OQ}\to E$ with the map $X_1(N)\to A_f$, we get the desired result.

		There is a fundamental difference between elliptic curves over $\Q$ and $\Q$-curves over bigger number fields: while, as we have seen, if $E$ is an elliptic curve over $\Q$, its $L$-function coincides with the $L$ function of some newform $f$, the same is not true in general for $\Q$-curves. In order to generalize the method used in section \ref{curvesoverQ}, one would like to be able to relate the $L$-function of $E$ to the $L$-function of a cuspform. This motivates the following definition, given in \cite{guique}.
			\begin{definition}
				An abelian variety $A$ over a number field $K$ is said \emph{strongly modular} if $L(A/K,s)=\prod_{i=1}^tL(f_i,s)$ for some newforms $f_i\in S_2(\Gamma_1(N_i),\varepsilon_i)$.
		\end{definition}

		\begin{proposition}
		Let $A/K$ be a strongly modular abelian variety. Then the newforms $f_1,\ldots,f_n$ such that $\displaystyle L(A/K,s)=\prod_{i=1}^nL(f_i,s)$ are unique, up to reordering.
	\end{proposition}
		\begin{proof}

		Let $B=\res_{K/\Q}(A)$, so that $\displaystyle L(B/\Q,s)=\prod_{i=1}^nL(f_i,s)$. Let $\{g_1,\ldots,g_m\}$ be another set of newforms with $\displaystyle L(B/\Q,s)=\prod_{i=1}^mL(g_i,s)$. Since $\displaystyle\prod_{i=1}^nL(f_i,s)=\prod_{j=1}^mL(g_j,s)$, the Dirichlet series of the left hand side and the right hand side coincide (in a suitable right half-plane of convergence). Thus for every prime $p$,
		\begin{equation}\label{galois_reps}
		\sum_{i=1}^na_p(f_i)=\sum_{j=1}^ma_p(g_j),
		\end{equation}
		where $a_p(f_i)$ (resp. $a_p(g_j)$) is the $p$-th coefficient in the $q$-expansion of $f_i$ (resp. $g_j$).

		For every newform $f\in S_2(\Gamma_1(M),\varepsilon)$ and every prime $l$ not dividing $M$, we denote by $\rho_l(f)$ the $l$-adic Galois representations attached to $f$ (see \cite{del}). Recall that this is a $2$-dimensional, irreducible representation with values in a finite extension of $\Q_l$ with the property that
		$$\tr(\rho_l(f)(\fr_p))=a_p(f),\quad \mbox{for all primes }p\nmid Ml.$$
		Now let $N$ be the product of all primes dividing the levels of the $f_i$'s and $g_j$'s and let $l$ be a prime not dividing $N$. Equation \eqref{galois_reps} implies that
		$$\tr\left(\bigoplus_{i=1}^n\rho_l(f_i)\right)(\fr_p)=\tr\left(\bigoplus_{j=1}^m\rho_l(g_j)\right)(\fr_p)\quad \forall\, p.$$
		It is well-known that semisimple, finite-dimensional Galois representations in characteristic $0$ are completely determined up to isomorphism by their traces at $\fr_p$ for every $p$ in a set of density $1$ (see for example \cite[Lemma 3.2]{delser}). Thus, the representations $\displaystyle \bigoplus_{i=1}^n\rho_l(f_i)$ and $\displaystyle\bigoplus_{j=1}^m\rho_l(g_j)$ are isomorphic. By looking at the dimension, we have that $n=m$ necessarily. Moreover, since all the components are irreducible, we can assume up to reordering that
		$$\rho_l(f_i)\simeq\rho_l(g_i)\quad \forall \, i\in \{1,\ldots,n\}.$$
		It follows that for all $i\in\{1,\ldots,n\}$ we have that $a_p(f_i)=a_p(g_i)$ for all primes $p$. Now \cite[Theorem 4.6.19]{miy} shows that $f_i=g_i$. This is done first by showing that the levels of $f_i$ and $g_i$ coincide by looking at the functional equation of their $L$-functions and then using the multiplicity one principle for newforms.
	\end{proof}
		Strongly modular $\Q$-curves are characterized by the following theorem.
			\begin{theorem}[{{\cite[Theorem~5.3]{guique}}}]
				A $\Q$-curve $E$ completely defined over a Galois number field $L$ is strongly modular if and only if $\gal(L/\Q)$ is abelian and the cocycle $\xi_L(E)$ is symmetric, namely $c(g,h)=c(h,g)$ for every cocycle $c$ representing $\xi_L(E)$.
		\end{theorem}
		An immediate corollary of this theorem is that $\Q$-curves completely defined over quadratic fields are strongly modular, because every cocycle class in $H^2(C_2,\Q^*)$ is symmetric.

  		\section{Modular abelian varieties and building blocks}\label{buiblo}
  		Let $\displaystyle f=\sum_{n=1}^{+\infty}a_nq^n \in S_2(\Gamma_1(N),\varepsilon)$ be a newform. The number field 
  		generated by the Fourier coefficients of $f$ will be denoted by $F$. We say that
  		$f$ has CM if there exists a non-trivial Dirichlet character $\chi$ such that $a_p=\chi(p)a_p$ for almost all $p$.

        Suppose that $f$ does not have CM. Let $\Gamma$ be the set of embeddings $\gamma \colon F \to \C$ such that 
  		there exists a Dirichlet character $\chi_{\gamma}$ with $\gamma(a_p)=\chi_{\gamma}(p)a_p$ for almost all primes $p$. Note 
  		that $\chi_{\gamma}$ is unique if it exists because $f$ does not have CM. It is proved in \cite{rib2} that $\Gamma$ is an abelian 
  		subgroup of $\aut(F)$ whose fixed field $F^{\Gamma}$ is $\Q(a_p^2/\varepsilon(p))$ where $p$ runs over a set $S$ of primes 
  		not dividing $N$ and having density $1$.
			\begin{definition}
    			The number field $L\coloneqq \OQ^{\cap_{\gamma}\ker\chi_\gamma}$ is called the \emph{splitting field} of $f$.
		\end{definition}
  		The abelian variety $A_f$ attached to $f$ is $\Q$-simple and has dimension equal to $[F\colon \Q]$; moreover $F$ is isomorphic to $\en_{\Q}^0(A_f)$ via the map that associates $a_n$ to $T_n$ and $\varepsilon(d)$ to $\langle d \rangle$ for all primes $d\in (\Z/N\Z)^*$. It is proved in \cite{gola} that the splitting field of $f$ is the smallest field over which all endomorphisms of $A_f$ are defined. The abelian variety $A_f$ is isogenous over $L$ to the power of an absolutely simple abelian variety $B_f$, called a \emph{building block} of $A_f$. The dimension of a building block satisfies the equality $\dim B_f=t\cdot [F^{\Gamma}\colon\Q]$ where $t$ is the \emph{Schur index} of $A_f$; it can be either $1$ or $2$ depending on the splitting of the class in $H^2(F/F^{\Gamma},F^*)$ of the $2$-cocycle
		$$c\colon \gal(F/F^{\Gamma})\times\gal(F/F^{\Gamma})\to F^*$$
  		$$(\sigma,\tau)\mapsto \frac{g(\chi_{\sigma}^{-1})g(\s\chi_{\tau}^{-1})}{g(\chi_{\sigma\tau}^{-1})},$$
  		where $\displaystyle g(\chi)=\sum_{a=1}^M\chi(a)e^{\frac{2\pi ia}{M}}$ for a Dirichlet character $\chi$ with conductor $M$.
  		From now on, we will always assume that $B_f$ is an elliptic curve without CM, since this is the only case that we will deal with. This means that $F^{\Gamma}=\Q$, $F/\Q$ is abelian and the class of $c$ is trivial in $H^2(F/\Q,F^*)$. The curve $B_f$ is a $\Q$-curve. The number field $L$ is the smallest one over which all endomorphisms of $A_f$ are defined, and $B_f$ is $L$-isogenous to all its Galois conjugates.
  		Since the class of $c$ in $H^2(F/\Q,F^*)$ is trivial, there exists a splitting map $\beta \colon \gal(F/\Q)\to F^*$ such that $\displaystyle c(\sigma,\tau)=\frac{\beta(\sigma)\s\beta(\tau)}{\beta(\sigma\tau)}$. The map $\beta$ is not unique: any other splitting map differs from $\beta$ by a coboundary; if $\beta'$ is another splitting map then for some $a \in F^*$ we have $\beta'(\sigma)=\beta(\sigma)\s a/a$.
 		After having identified $H^0(J_1(N),\Omega^1_{\C})$ with $S_2(\Gamma_1(N))$ by pulling back via the composed map $\mathcal H^* \to X_1(N)\to J_1(N)$, the construction of the variety $A_f$ as a quotient of $J_1(N)$ induces an isomorphism
  		$$H^0(A_f,\Omega^1_{\C})\stackrel{\sim}{\to} \bigoplus_{\sigma\colon F \hookrightarrow \C}\C\cdot  \s f(q)\frac{dq}{q}.$$
  		From now on we will identify these two spaces and $H^0(A_f,\Omega^1_{\C})$ will be regarded as a subspace of $H^0(X_1(N),\Omega^1_{\C})\simeq H^0(J_1(N),\Omega^1_{\C})$.
		Now fix a splitting map $\beta$ for $c$. Then the following theorem holds:
  			\begin{theorem}[{{\cite[Theorem~2.1]{gola}}}]\label{betaend}
    				There exists an endomorphism $w_{\beta}\in \en_L(A_f)$ such that:
    					\begin{enumerate}
      					\item the abelian variety $B=w_{\beta}(A_f)$ is a building block of $A_f$;
      					\item if $\omega_B$ is a generator of $H^0(B,\Omega^1_{\C})$, then $w_{\beta}^*(\omega_B)$ belongs to the subspace of $H^0(A_f,\Omega^1_{\C})$ generated by 
							$$\sum_{\sigma\in \gal(F/\Q)}\frac{g(\chi_{\sigma}^{-1})}{\beta(\sigma)}\s f;$$
      					\item all building blocks are of the form $a(B)$ for $a\in F$.
  				\end{enumerate}
		\end{theorem}
  		Let 
  		$$\lambda=\sum_{\sigma \in\gal(F/\Q)}\frac{g(\chi_{\sigma}^{-1})}{\beta(\sigma)}\in \C.$$
  		This quantity is nonzero (see \cite[Lemma 3.1]{gola}). Then the normalized cuspform attached to $(E,\pi)$ is 
  		$$h_{w_{\beta}}\coloneqq \frac{1}{\lambda}(w_{\beta}^*(\omega))\coloneqq \sum_{n=1}^{+\infty}\lambda_nq^n\in S_2(\Gamma_1(N)).$$
  		It is proved in \cite{gola} that $\lambda_n\in L$ for all $n$.

	\section{Quadratic $\Q$-curves}\label{quacur}
			\begin{definition}
				A \emph{quadratic $\Q$-curve} is a $\Q$-curve over a quadratic number field.
		\end{definition}
  		From now on, $E$ will denote a quadratic $\Q$-curve without CM completely defined over $K=\Q(\sqrt{d})$, for $d$ a square-free integer different from $1$. Let $\Delta_K$ be the discriminant of $K$, let $\gal(K/\Q)\coloneqq \{1,\nu\}$ and let $\mu\colon E\to {}^{\nu}E$ be a $K$-isogeny. Finally, let $\n\mu\mu$ coincide with multiplication by $m\in \Z$.
		\begin{lemma}[Serre]
			If $\Delta_K<0$, then $m>0$.
		\end{lemma}
  		\begin{proof}
			Fix an embedding $K\hookrightarrow\C$, so that $\nu$ is the restriction of complex conjugation to $K$. If $\Lambda\sub \C$ is a lattice uniformizing $E$, the conjugate curve $\n E=\overline{E}$ is uniformized by $\overline{\Lambda}$. The map $\mu_{\nu}$ can be identified with multiplication on $\C$ by some complex $\alpha$. Thus $\n\mu$ is multiplication by $\overline{\alpha}$, and therefore $m=\alpha\overline{\alpha}>0$.
		\end{proof}
		We will exhibit in section \ref{exs} explicit examples of $\Q$-curves completely defined over real quadratic fields with positive and negative $m$, and of $\Q$-curves completely defined over imaginary quadratic fields, which necessarily have positive $m$.
		
  		Let $B=\res_{K/\Q}(E)$ denote the restriction of scalars of $E$. This is an abelian surface defined over $\Q$, so either $B$ is isogenous to a product of two elliptic curves, or it is a $\Q$-simple abelian variety. In the latter case, thanks to Theorem \ref{Qcurvesaremodular}, it follows that $B$ is isogenous to $A_f$ for some newform $f$ whose Fourier coefficients generate a quadratic field. It is clear from the proof of the theorem that $\en_{\Q}^0(B)$ is isomorphic to $\Q[x]/(x^2-m)$, so that $B$ is simple over $\Q$ precisely when $m$ is not a square. On the other hand, $m$ is a square precisely when $E$ is $K$-isogenous to an elliptic curve $E_0$ over $\Q$. In this case, it is well-known that
		$$L(E/K,s)=L(E_0/\Q,s)\cdot L(E_0^{(d)}/\Q,s),$$
		where $E_0^{(d)}$ denotes the quadratic twist of $E_0$ by $d$, and therefore
		we have that
		$$L(E/K,1)=L(E_0/\Q,1)\cdot L(E_0^{(d)}/\Q,1).$$
		Thus, in this case we can use the method of section \ref{curvesoverQ} to get an analogue for \eqref{stima}.
  			\begin{example}\label{example_over_Q}
      			To get an example of such a situation, start with an elliptic curve $E/\Q$ without CM such that the group $E(\Q)[2]$ has order $2$. Let $P$ be its generator. Then the other two $2$-torsion points $P_1,P_2$ will be defined over some quadratic extension $K/\Q$. Now let $\phi_i$ be the isogeny with kernel $\{O,P_i\}$ for $i=1,2$ and let $E_i=E/\ker \phi_i$. The curves $E_1,E_2$ are defined over $K$ and $E_2=\s E_1$, since $\ker \phi_1$ and $\ker\phi_2$ are Galois conjugate one to each other. Clearly there is an isogeny $\phi=\phi_2\circ\widehat{\phi_1}\colon E_1\to E_2$ which has degree $4$ and is defined over $K$. Thus the curve $E_1$ is a $\Q$-curve defined over $K$ but isogenous to an elliptic curve defined over $\Q$, namely $E$.
      			For an explicit example, consider the elliptic curve $E\colon y^2=(x-1)(x^2+1)$. One checks that $j(E)=128$, thus $E$ has no CM. Using the notation above we have $P=(1,0)$, $P_1=(i,0)$ and $P_2=(-i,0)$. Then $E_1$ is the elliptic curve defined over $\Q(i)$ with equation
				$$E_1\colon y^2=x^3-x^2+(10i+11)x+6i-23.$$
		\end{example}

		\section{The newform attached to $E$}\label{newform}
		From now on, we will assume that $m$ is not a perfect square. Let
		$$f=\sum_{n=1}^{+\infty}a_nq^n \in S_2(\Gamma_1(N),\varepsilon)$$
		be a newform such that $\res_{K/\Q}(E)=B$ is isogenous to $A_f$. The number field generated by the $a_n$'s will be denoted by $F=\Q(\sqrt{m})$. Note that all endomorphisms of $A_f$ are defined over $K$ because $\mu$ itself is, so $\en_{\OQ}^0(A_f)=\en_{K}^0(A_f)$. On the other hand, since $\en_{\Q}^0(A_f)\simeq F$, it follows that $K$ is the splitting field of $f$.
  		Following the convention of \cite{gola}, we will implicitly fix an embedding of $F$ into $\C$. Let $\gal(F/\Q)=\{1,\sigma\}$. Then there exists a unique Dirichlet character $\chi$ such that $\s a_p=\chi(p)a_p$ for all primes $p\nmid N$ (see \cite{gola}).
  		The field $K$ equals $\OQ^{\ker \chi}$, which implies that $\chi$ is the primitive quadratic character corresponding to $K$ via class field theory. This means that if $\Delta_K$ is the discriminant of $K$ then $\chi$ is the primitive quadratic character modulo $|\Delta_K|$ given by:
    		$$\chi(p)=\begin{cases}
				1 & \mbox{if $p$ splits in $K$} \\
				-1 & \mbox{if $p$ is inert in $K$}\\
				0 & \mbox{if $p$ ramifies in $K$.}\\
		\end{cases}$$
  		Now recall that for the coefficients of $f$ it holds (see for example \cite{rib3}) that $a_p=\overline{a_p}\varepsilon(p)$ for all primes $p\nmid N$.
  		If $F$ is quadratic real, $\varepsilon$ must be trivial since $f$ does not have CM.
  		If $F$ is quadratic imaginary, we have $\s a_p=\chi(p)a_p$ but $\s a_p=\overline{a_p}$ and therefore we get $\chi=\varepsilon^{-1}=\varepsilon$ as characters modulo $N$. Of course $\varepsilon$ needs not to be primitive, but it is defined modulo $N$ and $d_{K/\Q}$ divides $N$ (see equation \eqref{conductor} below). Thus $\varepsilon$ is just the composition of $\chi$ with the projection $(\Z/N\Z)^*\to (\Z/|\Delta_K|\Z)^*$.
	    In order to find the $q$-expansion of $f$, we will look at local factors of the $L$-functions.

		We remark that when $N$ is a prime congruent to $1$ modulo $4$, $\chi$ is the Legendre symbol modulo $N$ and $\displaystyle\Gamma_{\chi}(N)=\left\{\barr{cc}a & b\\c & d\narr\in \Gamma_0(N)\colon \chi(d)=1\right\}$, in \cite[section 5]{cremona} the author systematically computes $q$-expansions of newforms in $S_2(\Gamma_{\chi}(N))$ whose Fourier coefficients generate a quadratic field. The abelian surface attached to such a newform splits over $\Q(\sqrt{N})$ as a product of two $\Q$-curves with everywhere good reduction. Even though there are some similarities between our computations and those of \cite{cremona}, our result is different in nature, since we start with a $\Q$-curve and then compute the Fourier coefficients of the attached newform, while in \cite{cremona} the author starts with a newform and then computes a corresponding $\Q$-curve. Moreover, we have no assumptions on $N$.

\subsection{Local factors of $L$-functions}

		Let $\lambda$ be a finite place of~$F$, and let $l$ be its residue characteristic.  Let $V_l$ be the $l$-adic Tate module of $A_f$; this is a free module of rank~2 over the ring $\Q_l\otimes_\Q F$ with an action of $G_\Q$.  We consider
		$$V_\lambda = F_\lambda\otimes_{\Q_l} V_l;$$
		this is a $2$-dimensional $F_\lambda$-linear representation of~$G_\Q$.

		Let $p$ be a prime number different from~$l$, and let $D_p$ and $I_p$ be a decomposition group at~$p$ and the corresponding inertia group, respectively.  Let $(V_\lambda)_{I_p}$ be the space of coinvariants of $V_\lambda$ under~$I_p$.

		The $L$-factor of $f$ at~$p$ is of the form
		$$L_p(f,s) = P_p(f, p^{-s})$$
		where
		$$P_p(f, x) = 1 - a_p x + \varepsilon(p) p x^2 \in F[x].$$
		and we let $\varepsilon(p)=0$ if $p\mid N$.
		Then we have
		\begin{equation}
		\label{charpoly}
		\det_{F_\lambda}({\id} - x\cdot\fr_p \mid (V_\lambda)_{I_p}) = P_p(f,x)
		\end{equation}
		(see for example \cite[Theorem 4]{roh} and \cite{car86}).

		On the other hand, for every prime ideal $\p\sub \mathcal O_K$ lying over a rational prime $p$ and every prime $l\neq p$, the absolute Galois group $G_K\coloneqq \gal(\overline{K}/K)$ acts on the $l$-adic Tate module $V_l$ of $E$ yielding a $2$-dimensional $l$-adic Galois representation of $G_K$. Let $D_{\p}\sub G_K$ denote a decomposition group at $\p$, $I_{\p}$ the corresponding inertia subgroup and $\fr_{\p}\in D_{\p}$ any Frobenius element at $\p$. Then if $E$ has good reduction at $\p$ the characteristic polynomial of $\fr_{\p}$ is given by
		$$P_{\p}(E,x)=1-c_{\p}x+N_{K/\Q}(\p)x^2\in \Z[x],$$
		where $c_{\p}=N_{K/\Q}(\p)+1-|E(\F_{\p})|$. If $E$ has bad reduction at $\p$, we set
		$$P_{\p}(E,x)=\begin{cases}1 & \mbox{if $E$ has additive reduction}\\1-x & \mbox{if $E$ has split multiplicative reduction}\\1+x &\mbox{if $E$ has non-split multiplicative reduction.}\end{cases}$$
		The $L$-factor at $\p$ is given by
		$$L_{\p}(E,s)=P_{\p}(E,N_{K/\Q}(\p)^{-s}).$$
		The following equation holds for all $\p$:
		$$\det_{\Q_l}({\id}-x\cdot \fr_{\p}\mid (V_l)_{I_{\p}})=P_{\p}(E,x).$$
		By \cite[Proposition~3]{mil1} and \cite[Theorem~5]{roh}, the following equalities hold for all rational primes $p$:
		\begin{equation}\label{Lfuncts}
			\prod_{\p\mid p}L_{\p}(E/K,s)=L_p(A_f/\Q,s)=\prod_{\vartheta \in \gal(F/\Q)}L_p({}^{\vartheta\!}f,s),
		\end{equation}
		where $L_p(A_f/\Q,s)$ is the local factor at $p$ of $L$-function attached to the Galois representation on the $l$-adic Tate module of $A_f$.

		The equality \eqref{Lfuncts} will be used to compute the $a_p$'s. In order to do that, we will separately analyze primes of good and bad reduction for $E$.
		According to \cite[Proposition~1]{mil1}, if $\mathcal N_{K}(E)$ is the conductor of $E$, then one has that
		$$N_{K/\Q}(\mathcal N_{K}(E))\Delta_K^2=\mathcal N_{\Q}(\res_{K/\Q}(E))$$
		and combining this with the fact that $\mathcal N_{\Q}(A_f)=N^2$ (see \cite{car}) we get the following formula:
		  \begin{equation}\label{conductor}
			  N_{K/\Q}(\mathcal N_{K}(E))\Delta_K^2=N^2.
		\end{equation}
	  Therefore primes of bad reduction for $A_f$ are exactly primes lying under primes of bad reduction for $E$ and primes which ramify in $K$.
		  \begin{lemma}\label{conductor2}
			    The conductor of $E$ is a principal ideal generated by an integer. Moreover, $E$ has bad reduction at a prime $\p$ if and only if it has bad reduction at $\n\p$.
		\end{lemma}
	  \begin{proof}
	    Let $\mathcal N_K(E)=\p^rI$, where $\p$ is a prime ideal of $K$, $r\in \N$ and $I$ is an ideal of $K$ which is coprime with $\p$. We will show that either $\p^r$ is a principal ideal generated by $p^k$ for some $k$, where $p$ is the rational prime lying under $\p$, or $\n \p^r$ exactly divides $I$.

	    Suppose that $\p$ lies above a ramified rational prime $p$. Then $N_{K/\Q}(\p)=p$. Since $\p$ is the only prime lying above $p$ and equation \eqref{conductor} implies that $N_{K/\Q}(\mathcal N_{K}(E))\Delta_K^2$ must be a square in $\Z$, this means that $\p$ has to divide the conductor of $E$ an even number of times, say $2k$ times. This implies $r=2k$ and $\p^r=(p^k)$.

	    Suppose now that $\p$ lies over an inert prime $p$. This amounts to saying that $\p=(p)$, implying that $\p^r=(p)^r$.

	    Finally, suppose that $\p$ lies over a split prime $p$. Since $E$ is $K$-isogenous to $\n E$, the two curves have the same conductor. But $\mathcal N_K(\n E)=\n\mathcal N_K(E)$ and this implies that $\n\p^r$ exactly divides $\mathcal N_K(\n E)$ and hence also $\mathcal N_K(E)$, concluding the proof.
	\end{proof}
	Thanks to the above lemma, by a small abuse of notation we can say that $E$ has bad (good) reduction at a rational prime $p$.
	
	  \subsection{Primes of good reduction for $E$}
	  Let $p$ be a prime of good reduction for $E$. Equation \eqref{Lfuncts} becomes:
	  \begin{equation}\label{goodred}
			    \prod_{\p\mid p}(1-c_{\p}N(\p)^{-s}+N(\p)^{1-2s})=(1-a_pp^{-s}+\varepsilon(p)p^{1-2s})(1-\s a_pp^{-s}+\s\varepsilon(p)p^{1-2s}).
		\end{equation}
	  \subsubsection*{Ramified case}
	  Suppose $p$ is a rational prime ramified in $K$. In this case, $p$ divides $N$ because of \eqref{conductor}, and equation \eqref{goodred} therefore becomes
	  $$1-c_{\p}p^{-s}+p\cdot p^{-2s}=1-(a_p+{}^{\sigma}a_p)p^{-s}+(a_p\cdot\s a_p)p^{-2s},$$
	  where $\p$ is the unique prime of $K$ lying over $p$. We then have
	  $$\begin{cases}
      a_p+{}^{\sigma}a_{p}=c_{\p} & \\
      a_p\cdot{}^{\sigma}a_p=p, & \\
	\end{cases}$$
	  implying
	  $$a_p=\frac{c_{\p}\pm\sqrt{c_{\p}^2-4p}}{2}.$$
	  In particular, $c_{\p}-4p$ is a square in $F$.
	  \subsubsection*{Inert case}
	  If $p$ is inert in $K$ and $\p$ is the unique prime lying above it, we get
		$$1-c_{\p}p^{-2s}+p^{2-4s}=1-(a_p+{}^{\sigma}a_p)p^{-s}+(2\varepsilon(p)p+a_p\cdot \s a_p)p^{-2s}+$$
	  $$+(a_p+{}^{\sigma}a_p)p\cdot p^{-3s}+p^{2-4s},$$
	  which leads to the following system:
	  $$\begin{cases}
      a_p+{}^{\sigma}a_p=0 & \\
      2\varepsilon(p)p+a_p\cdot {}^{\sigma}a_p=-c_{\p}. & \\
	\end{cases}$$
	  Thus $a_p=\pm\sqrt{c_{\p}+2\varepsilon(p)p}$ and in particular $c_{\p}+2\varepsilon(p)p$ is a square in $F$. Here $\varepsilon(p)=1$ if $F$ is real and $\varepsilon(p)=-1$ if $F$ is imaginary.
	  \subsubsection*{Split case}
	  If $p$ is split in $K$, then $\varepsilon(p)=1$, there are two primes $\p_1,\p_2$ lying over $p$ and $\p_2={}^{\nu}\p_1$. Let $l$ be a prime different from $p$. Since $E$ and $\n E$ are isogenous, the two Tate modules $T_l(E)$ and $T_l(\n E)$ are isomorphic as $G_K$-modules. This shows that $c_{\p_i}(E)=c_{\p_i}(\n E)$ for $i=1,2$. Now we claim that $c_{\p_1}(E)=c_{\p_2}(\n E)$. To see this, let $\overline{\nu}\in \GQ$ be any lift of $\nu$, and let
	  $$c_{\overline{\nu}}\colon G_K\to G_K$$
	  $$\tau\mapsto \overline{\nu}\tau\overline{\nu}^{-1}$$
	  be the conjugation by $\overline{\nu}$. This is a well-defined homomorphism because $G_K$ is normal in $\GQ$.
	  Now let
		$$\varphi_{\overline{\nu}}\colon \aut(T_l(E))\to \aut(T_l(\n E))$$
		$$f\mapsto (x\mapsto {}^{\overline{\nu}}f(\leftidx{^{\overline{\nu}^{-1}}\!}x)).$$
	Then it is easy to check that the following diagram commutes:
	  $$\begin{xymatrix}{
			   G_K\ar[r]^{c_{\nu}}\ar[d] & G_K \ar[d]\\
				\aut(T_l(E))\ar[r]^{\varphi_{\overline{\nu}}} & \aut(T_l(\n E))
			  }\end{xymatrix}$$
	where the two vertical arrows are the usual $l$-adic representations of $G_K$. If $\fr_{\p_i}\in G_K$ is a Frobenius at $\p_i$ for $i=1,2$, it is clear that $c_{\overline{\nu}}(\fr_{\p_1})=\fr_{\p_2}$. On the other hand, if one chooses a $\Z_l$-basis $\{\mathbf{e_1},\mathbf{e_2}\}$ for $T_l(E)$, then $\{{}^{\overline{\nu}\!}\mathbf{e_1},{}^{\overline{\nu}\!}\mathbf{e_2}\}$ is a $\Z_l$-basis for $T_l(\n E)$ and the map $\varphi_{\overline{\nu}}$ written with respect to these bases is just the identity. This shows that the characteristic polynomial of $\fr_{\p_1}$ acting on $T_l(E)$ coincides with the characteristic polynomial of $\fr_{\p_2}$ acting on $T_l(\n E)$, and the claim follows.

	 By the discussion above, we have that $c_{\p_1}(E)=c_{\p_2}(E)$, so that we can just write $c_{\p}$ for that. Equation \eqref{goodred} reads:
	  $$(1-c_{\p}p^{-s}+p^{1-2s})^2=(1-a_pp^{-s}+p^{1-2s})(1-{}^{\sigma}a_pp^{-s}+p^{1-2s}),$$
	  leading to
	  $$\begin{cases}
      a_p+{}^{\sigma}a_p=2c_{\p} & \\
      a_p\cdot \s a_p=c_{\p}^2. & \\
	\end{cases}$$
	  Therefore $a_p=c_{\p}$ and $\s a_p=a_p$.

	  \subsection{Primes of bad reduction for $E$}

	Let $p$ be a prime of bad reduction for $E$. Then $\varepsilon(p)=0$. Equation \eqref{Lfuncts} becomes:
	  \begin{equation}\label{badred}
	    \prod_{\p\mid p}(1-c_{\p}N(\p)^{-s})=\prod_{\sigma\in\gal(F/\Q)}(1-{}^{\sigma}a_pp^{-s}),
	\end{equation}
	  where $c_{\p}= 1,-1,0$ if $E$ has split multiplicative, non-split multiplicative or additive reduction at $\p$, respectively.

	  \subsubsection*{Ramified case}
	  Let $\p$ be the unique prime lying above $p$. In the proof of Lemma \ref{conductor2}, we showed that $\p$ has to divide the conductor of $E$ an even number of times, and so the reduction at $\p$ must be additive. This implies $c_{\p}=a_p=0$.
	  \subsubsection*{Inert case}
	  Let $p$ be inert in $K$ and let $\p$ be the unique prime lying above $p$. Then
	  $$1-c_{\p}p^{-2s}=1-(a_p+{}^{\sigma}a_p)p^{-s}+a_p\cdot \s a_pp^{-2s}.$$
	  Therefore $a_p=c\sqrt{m}$ for some $c\in \Z$ and $c^2m=c_{\p}$. Since $|c_{\p}|\leq 1$, if $|m|>1$, then we must have $c_{\p}=a_p=c=0$. Otherwise, namely if $m=-1$, we must have either $c_{\p}=a_p=c=0$ or $c_{\p}=-1$, $c\in\{\pm1\}$ and $a_p=c\sqrt{-1}$.
	  \subsubsection*{Split case}
	  Let $\p_1,\p_2$ be the primes lying above $p$. The same argument we used for the split case applies again, just noticing that since $T_l(E)\simeq T_l(\n E)$ as $G_K$-modules, we have $T_l(E)_{I_{\p_1}}\simeq T_l(\n E)_{I_{\p_1}}$ as $D_{\p_1}$-modules, where $D_{\p_1}\sub G_K$ is any decomposition group for $\p_1$. Therefore we get $c_{\p_1}=c_{\p_2}$ and consequently
	  $$1-2c_{\p_1}p^{-s}+c_{\p_1}^2p^{-2s}=1-2a_pp^{-s}+\s a_p\cdot a_pp^{-2s},$$
	  so that $a_p=c_{\p_1}=c_{\p_2}$.

\subsection{Finding the sign of the $a_p$'s}

		Now given $E$, we have to decide for each $p$ which coefficient to choose between $a_p$ and $\s a_p$ when $p$ is a ramified or inert prime of good reduction for $E$ or, when $m=-1$, $p$ is inert and $E$ has non-split multiplicative reduction at $p$. The ambiguity arises from the fact that, unlike in the setting of elliptic curves over $\Q$ to which we can associate a unique newform, here we associate to $E$ a pair of conjugate newforms, which a priori we cannot distinguish one from another. The object we can easily compute starting from $E$ is a family of pairs $\{(a_p,\s a_p)\}_p$. We will see that the choice of a square root of $m$ will allow us to pick, for every prime $p$, exactly one between $a_p$ and $\s a_p$ so that the collection of all the picked coefficients will coincide with the collection of the Fourier coefficients of prime index of a newform $f$. Choosing the other square root of $m$ will give us the Fourier coefficients of $\s f$. To start, recall that $F=\Q(a_n\colon n\in \N)=\Q(\sqrt{m})$ acts on $B$, where for every prime $p$ the coefficient $a_p$ acts as $T_p$ does. Note that $\mu\colon E\to \n E$ induces an endomorphism $\mu_*\colon B\to B$ such that $(\n\mu)_*\circ\mu_*=m$. Furthermore, we have $(\n\mu)_*=\mu_*$. Therefore the choice of a square root of $m$ induces an inclusion $\Z[\mu_*]\sub F$. From now on, we fix such an embedding. This will correspond to the choice of one newform in the Galois orbit of $f$. Now note that the ring $\en_K(B_K)$ can be identified with the ring $\displaystyle \barr{cc} \en_K(E) & \Hom_K(\n E, E)\\ \Hom_K(E,\n E) & \en_K(\n E)\narr$, whose elements are $2\times 2$ matrices $(a_{ij})$ whose entries lie in the corresponding set of isogenies; for example, $a_{11}\in \en_K(E)$. Under this identification, the subring $\Z[\mu_*]$ corresponds to the subring $\displaystyle \Z\cdot \barr{cc} 0 & \n\mu\\ \mu & 0\narr$.

		Before starting to analyze each problematic case, let us recall two fundamental facts. If $l$ is any prime, then:
		\begin{enumerate} 
			\item[i)] there is a canonical isomorphism of $\Z_l[\GQ]$-modules
				\begin{equation} \label{indrep}
					T_l(B)\simeq T_l(E)\otimes_{\Z_l[G_K]}\Z_l[\GQ]
			\end{equation}
			(see \cite{mil1});
			\item[ii)] let $p$ be a prime different from $l$. Then there is an isomorphism of $\gal(\overline{\F_p}/\F_p)$-modules
				\begin{equation}\label{inertia-invariants}
				(T_l(B))^{I_p}\to T_l(B_{\F_p})
			\end{equation}
			(see for example \cite{serta}).
		\end{enumerate}

\subsubsection*{Inert primes of good reduction}

		 Let $p$ be an inert prime of good reduction for $E$, and let $\p$ be the unique prime of $K$ lying above $p$. In view of \eqref{conductor}, $B$ has good reduction at $p$. Now fix a prime $l\neq p$. Then the Tate module $T_l(B)$ is unramified at $p$. Since $F$ acts $\Q_l$-linearly on $T_l(B)$, the Tate module is also a $F\otimes \Q_l$-module, and one can show that $T_l(B)$ has dimension $2$ as an $F\otimes \Q_l$-module (see for example \cite{dis}). Moreover, any Frobenius at $p$ satisfies the equation
		 $$x^2+a_px+\varepsilon(p)p=0$$
		in $\en_{F\otimes \Q_l}(T_l(B))$, where $a_p$ is viewed as an endomorphism of $T_l(B)$. Since $T_l(B)$ is isomorphic to $T_l(B_{\mathbb F_p})$ as a $\gal(\overline{\F_p}/\F_p)$-module, any Frobenius at $p$ must satisfy the same equation in $\en_{F\otimes\Q_l}(T_l(B_{\F_p}))$. Now note that the Frobenius at $p$ acts on $B_{\F_p}$ as the matrix $\barr{cc}0 & \n\fr_p\\ \fr_p & 0\narr$, where $E_{\p}$ is the reduction of $E$ modulo $\p$ and $\fr_p\colon E_{\p}\to \n E_{\p}$ maps $(x,y)$ to $(x^p,y^p)$ and $\n\fr_p\colon \n E_{\p}\to E_{\p}$ is defined analogously. From the computations of the coefficients of $f$ performed above, we see that there exists an integer $c$ such that $a_p=c\sqrt{m}$, so that $a_p$ acts on $B$ as $\mu$ followed by multiplication by $c$. Therefore $a_p$ acts on $B_{\F_p}$ as the matrix $\barr{cc}0 & c\,\n\mu_{\p}\\c\mu_{\p} & 0 \narr$, where $\mu_{\p}$ is the reduction of $\mu$ modulo $\p$. Hence the following must hold:
		 $$\barr{cc}\n\fr_p\circ\fr_p & 0\\0 & \fr_p\circ\n\fr_p\narr-\barr{cc}c\,\n\mu_{\p}\circ\fr_p & 0\\0 & c\mu_{\p}\circ\n\fr_p\narr+\barr{cc}\varepsilon(p)p & 0\\0 & \varepsilon(p)p\narr=0,$$
		implying that
		\begin{equation*}
		\fr_{p^2}-c\,\n\mu_{\p}\circ\fr_p+\varepsilon(p)p=0
		\end{equation*}
		on $E_{\p}$, where $\fr_{p^2}$ is the usual Frobenius. What we know is the absolute value of $c$, we have to decide the sign. Let $T=\fr_{p^2}-|c|\,\n\mu_{\p}\circ\fr_p+\varepsilon(p)p$ and $S=\fr_{p^2}+|c|\,\n\mu_{\p}\circ\fr_p+\varepsilon(p)p$. Let $Q$ be a point on $E_{\p}$. Note that if $T(Q)=S(Q)$ then $(S-T)(Q)=0$, so $(2|c|\,\n\mu_{\p}\circ\fr_p)(Q)=0$, which implies that $Q$ has order dividing $2p|mc|$. Therefore if $Q$ has order coprime with $2p|mc|$ then $T(Q)\neq S(Q)$. If $T(Q)=0$ then $c$ is positive, otherwise it is negative. This gives us an algorithm to decide the sign of $c$. Note that if we had fixed the other embedding $\Z[\mu_*]\to \Q(\sqrt{m})$, we would get the opposite sign.

\subsubsection*{Ramified primes of good reduction}

		Let now $p$ be a ramified prime of good reduction for $E$, and let $\p$ be the unique prime of $\mathcal O_K$ lying above it. Then we have the following lemma.
			\begin{lemma}
				There is an exact sequence of group schemes over $\F_p$:
				$$0\to \mathbb G_{a}\to B_{\F_p}\to E_{\p}\to 0.$$
		\end{lemma}
			\begin{proof}
				The proof is essentially an application of the results of \cite{edi2}. In the notation of section $5$ of that paper, the discrete valuation ring $D$ is $\Z_{(p)}$, i.e. the localization of $\Z$ with respect to the prime ideal $(p)$, while the discrete valuation ring $D'$ is $\mathcal O_{K,\p}$. The abelian variety $X$ is $\res_{\mathcal O_{K,\p}/\Z_{(p)}}\mathcal E$, where $\mathcal E$ is the N\'{e}ron model of $E$ over $\mathcal O_{K,\p}$. By \cite[Proposition 6, section 7.6]{bos}, $\res_{\mathcal O_{K,\p}/\Z_{(p)}}\mathcal E$ is the N\'{e}ron model of $B$ over $\Z_{(p)}$, which we denote by $\mathcal B$. Now again in \cite[section 5.2]{edi2}, the author constructs a filtration $\mathcal B_{\F_p}= F^0\mathcal B_{\F_p}\supseteq F^1\mathcal B_{\F_p}\supseteq F^2\mathcal B_{\F_p}=0$ where for any $\F_p$-algebra $C$ and $i=0,1,2$ one has
				$$(F^i\mathcal B_{\F_p})(C)=\ker (\mathcal B_{\F_p}(C)\stackrel{\sim}{\to} \mathcal E(C[t]/(t^2))\to \mathcal E(C[t]/(t^i))).$$
				The successive quotients of the filtration $\gr \mathcal B_{\F_p}\coloneqq F^i\mathcal B_{\F_p}/F^{i+1}\mathcal B_{\F_p}$ give rise to a short exact sequence
				$$0\to F^1\mathcal B_{\F_p}\to \mathcal B_{\F_p}\to \gr^0\mathcal B_{\F_p}\to 0.$$
				Now $\gr^0\mathcal B_{\F_p}=\mathcal B_{\F_p}/F^1\mathcal B_{\F_p}\simeq \mathcal E_{\p}$, while $\gr^1\mathcal B_{\F_p}=F^1\mathcal B_{\F_p}$ by the fact that $F^2\mathcal B_{\F_p}=0$. The isomorphism (5.1.2) in \cite{edi2} tells us that
				$$\gr^i\mathcal B_{\F_p}\stackrel{\sim}{\to}T_{\mathcal E_{\p},0}\otimes_{\F_p}(m/m^2)^{\otimes i}$$
				as group schemes, where $m$ is the maximal ideal of $\mathcal O_{K,\p}$. This shows that $\gr^0\mathcal B_{\F_p}\simeq \mathbb G_{a}$, proving the claim.
\end{proof}

		By equation \eqref{charpoly}, we notice that $a_p$ is the trace of $\fr_p$ on $(V_{\lambda})_{I_p}=(V_l(B)\otimes_{\Q_l} F_{\lambda})_{I_p}$. The isomorphism \eqref{indrep} shows that we can find a basis of $T_l(B)$ such that $I_p$ acts via matrices of the form $\barr{cccc}1 & 0 & 0 & 0\\0 & 1 & 0 & 0 \\ 0 & 0 & * & *\\ 0 & 0 & * & *\\\narr$. Therefore $(T_l(B))_{I_p}\simeq (T_l(B))^{I_p}$. Now the lemma above implies, since $\mathbb G_{a}$ is $l$-torsion free, that $(T_l(B))^{I_p}\otimes_{\Z_l}\Q_l$ is a $1$-dimensional $F$-linear representation of $D_p$, and therefore $\fr_p$ acts as multiplication by $a_p$ on $(T_l(B))^{I_p}$.
		Now we can use the same argument we used above, together with \eqref{inertia-invariants}, to see that $\fr_p\colon  E_{\p}\to \n E_{\p}$ equals the map $c\cdot \mu_{\p}$, where $c$ is an integer of which we know the absolute value but not the sign. This sign can be determined in an analogous way as in the inert case. 

		\subsubsection*{Inert primes of multiplicative reduction}

		We now consider the case where $p$ is a prime that is inert in~$K$ such that $E$ has non-split multiplicative reduction at~$p$.  This case only occurs if $m=-1$, and we will exhibit in section \ref{exs} an example of a $\Q$-curve over $\Q(\sqrt{17})$ which has non-split multiplicative reduction at $5$.

		Let $\p$ be the unique prime of~$K$ lying over~$p$. Let $G$ be the smooth locus of the reduction of~$E$ modulo~$\p$; this is a non-split torus of dimension~$1$ over $k(\p)$.  Then we have a canonical isomorphism
		$$B_{\F_p} \simeq \res_{k(\p)/\F_p} G.$$
		We note that reducing $\mu$ gives an isomorphism
		$$
		\mu_{\p}\colon G\to \n G,
		$$
		where $\n G$ is the Galois conjugate of $G$ via $\nu$, which reduces to the Frobenius on $k(\p)$, such that if $\n\mu_{\p}\colon \n G\to G$ is the conjugate of $\mu_{\p}$, then we have $\n\mu_{\p}\circ\mu_{\p}=-1$.


		The $L$-factor we are trying to determine is $1-a_p p^{-s}$.  Recall that in our setting, we have $a_\p=-1$ and $a_p = c\sqrt{-1}$ for some unknown $c\in\{\pm1\}$.  This $a_p$ arises as the eigenvalue of $\fr_p$ on $(V_\lambda)_{I_p}$ by \eqref{charpoly}. Since $E$ has semi-stable reduction at $\p$, the inertia subgroup at $\p$ acts unipotently on $V_l(E)$. Using \eqref{indrep}, one can check that $I_p$ acts unipotently on $V_{\lambda}$. Therefore there is a short exact sequence
		$$
		0\longrightarrow(V_\lambda)^{I_p}\longrightarrow V_\lambda
		\longrightarrow(V_\lambda)_{I_p}\longrightarrow 0.
		$$
		The product of the eigenvalue of $\fr_p$ on $(V_\lambda)_{I_p}$ and the eigenvalue of $\fr_p$ on $(V_\lambda)^{I_p}$ equals $-p$, because the determinant of the Galois representation on $V_{\lambda}$ equals $\chi$ times the $l$-adic cyclotomic character $\chi_l$ (see for example \cite[Proposition 2.2]{rib3}) and $\chi(p)=-1$ and $\chi_l(p)=p$. Therefore the eigenvalue of $\fr_p$ on $(V_\lambda)^{I_p}$ equals $-p/a_p=-p/(c\sqrt{-1})=cp\sqrt{-1}$.  This implies that the Frobenius endomorphism $\fr_p$ of $B_{\F_p}$ equals $cp(\mu_\p)_*$.
		This $\fr_p$ is induced by the endomorphism of $G\times \n G$ defined by the matrix $\barr{cc} 0& \fr_p \\ \fr_p & 0\narr$.
		Furthermore, the endomorphism $(\mu_\p)_*$ of $B_{\F_p}$ is induced by the endomorphism of $G\times \n G$ defined by the matrix $\barr{cc} 0& \n\mu_\p \\ \mu_\p & 0\narr$.
		This implies that the two maps $\fr_p$ and $cp\mu_\p$ from $G$ to $\n G$ are equal. Hence we can determine $c$ by taking random points $Q$ of~$G$ and checking for which $c$ we have the equality $\fr_p(Q)=cp\mu_\p(Q)$ in $\n G$.

	\section{Computing $L(E,1)$}\label{mainsect}

		We mentioned in the previous section that there is an equality of $L$-functions
		$$L(E/K,s)=L(f,s)\cdot L(\s f,s).$$
		From this we can derive the formula
			\begin{equation}\label{Lvaluequadratic}
		L(E/K,1)=\left(-2\pi i\int_0^{i\infty}f(t)dt\right)\cdot \left(-2\pi i\int_0^{i\infty}\s f(t)dt\right).
		\end{equation}
		The key point now is that there exists a map $\pi\colon A_f\to E$ defined over $K$ (and consequently there exists a conjugate map $\n \pi\colon A_f\to \n E$). In fact if $w_{\beta}$ is an endomorphism of $A_f$ as in Theorem \ref{betaend}, then there is an isogeny $\varphi\colon w_{\beta}(B)\to E$, and we can let $\pi\coloneqq \varphi\circ w_{\beta}$. This allows us to consider the pullback of an invariant differential $\omega_E$ on $E$, but this time this pullback needs not to be a multiple of the cusp form $f$. In fact, in the notation of section \ref{buiblo}, by Theorem \ref{betaend} the space generated by $\pi^*(\omega_E)$ is spanned by $h=h_{\beta}$ and $\n h=\n(h_{\beta})$.
		First of all we need to understand how to change the base of $H^0(A_f,\Omega^1_{\C})$ from $\{f,\s f\}$ to $\{h,\n h\}$. This is easily done by just looking at the definitions. Recall the following standard result:
		$$g(\chi)=\begin{cases}\sqrt{\Delta_K} & \mbox{if } \Delta_K>0\\ i\sqrt{-\Delta_K} & \mbox{if } \Delta_K<0\end{cases}$$
		(for a proof see for example \cite{lan}). From now on, when we will write $\sqrt{\Delta_K}$ we will mean one of the two values given above, according to the sign of $\Delta_K$.
		Set $\displaystyle \kappa\coloneqq \frac{\sqrt{\Delta_K}}{\beta(\sigma)}\in FK$. Now we have to be a bit careful in distinguishing two cases, namely $K=F$ and $K\neq F$ (as we will see in section \ref{exs}, both cases actually occur). In the first case we can identify the two Galois groups $\gal(K/\Q)$ and $\gal(F/\Q)$ so that $\nu=\sigma$. In the second one we can identify $\gal(FK/\Q)$ with the group $\{1,\sigma,\nu,\sigma\nu\}$, where by a small abuse of notation $\sigma$ generates $\gal(FK/K)$ and $\nu$ generates $\gal(FK/F)$. Recall that according to the definition of the cocycle $c$ given in section \ref{buiblo} we have		
		$$c(\sigma,\sigma)=\frac{g(\chi^{-1})g(\s\chi^{-1})}{g(\textbf{1}_{\Delta_K})}=g(\chi)^2=\s\beta(\sigma)\beta(\sigma).$$
		This implies $\kappa\cdot\s\kappa=\eta$, where $\eta=-1$ if $K=F$ and $\eta=1$ if $K\neq F$. Note also that in this last case $\n\kappa=-\kappa$. Then by definition
		$$h=\frac{1}{1+\kappa}(f+\kappa \s f)=\frac{(1+\s\kappa)f+(\eta+\kappa)\s f}{(1+\kappa)(1+\s\kappa)}.$$
	    A priori $h$ has coefficients in the composite field $FK$ but it is clear that as long as $\eta=1$ one has that $\s h=h$, which implies that the coefficients of $h$ lie in fact in $K$, as predicted by the theory. This is trivially true in the case where $K=F$ and $\eta=-1$. One checks easily that
	    $$\barr{c}h\\ \n h\narr=\barr{cc}\dfrac{1}{1+\kappa} & \dfrac{1}{1+\eta\cdot\s\kappa}\\ \dfrac{1}{1-\kappa} & \dfrac{1}{1-\eta\cdot\s\kappa}\narr\barr{c}f\\ \s f\narr.$$
		Note that the determinant of the transformation equals $\displaystyle \frac{2}{\kappa-\eta\cdot \s \kappa}$, so it is always non-zero.
	    Inverting the system leads to
	    $$\barr{c}f\\ \s f\narr=\frac{1}{2}\barr{cc}1+\kappa & 1-\kappa\\ 1+\eta\cdot\s\kappa & 1-\eta\cdot\s\kappa\narr\barr{c}h\\ \n h\narr.$$
	    This gives
	    $$\begin{dcases}L(f,1)=\frac{1}{2}(1+\kappa)L(h,1)+\frac{1}{2}(1-\kappa)L(\n h,1) & \\ L(\s f,1)=\frac{1}{2}(1+\eta\cdot\s\kappa)L(h,1)+\frac{1}{2}(1-\eta\cdot\s\kappa)L(\n h,1), & \end{dcases}$$
		and substituting in equation \eqref{Lvaluequadratic} we get
			\begin{equation}\label{Lvaluequadratic2}
				L(E/K,1)=\frac{(1+\kappa)(1+\eta\cdot\s\kappa)}{4}L(h,1)^2+\frac{(1-\kappa)(1-\eta\cdot\s\kappa)}{4}L(\n h,1)^2.
		\end{equation}
		Note that both the element $\kappa\in FK$ and the cusp form $h$ depend on the chosen splitting map $\beta$. We will now explain how to make a choice for the splitting map which will allow us to simplify equation \eqref{Lvaluequadratic2}.

		Let $\beta$ be a splitting map. Then the element $\beta(\sigma)$ satisfies $N_{F/\Q}(\beta(\sigma))=\Delta_K$. This proves that $m\in N_{K/\Q}(K)$, so let $\alpha\in K$ be an element of norm $m$. Let $\omega_E$ be an invariant differential on $E$ and let $\omega_{\n E}'$ be an invariant differential on $\n E$. Let $\vartheta\in K$ be such that $\mu^*(\omega_{\n E}')=\vartheta\cdot\omega_E$ and set $\displaystyle \omega_{\n E}\coloneqq \frac{\vartheta}{\alpha}\omega_{\n E}'$. Then $\mu^*(\omega_{\n E})=\alpha\cdot\omega_E$. Let now $\alpha'\in K$ be such that $(\n\mu)^*(\omega_E)=\alpha'\omega_{\n E}$. Since $(\n\mu\circ \mu)^*$ coincides with multiplication by $m$, this shows that $\alpha'=\n\alpha$. To find explicitly such an $\alpha$, choose a Weierstrass equation for $E$ of the form $y^2+a_1xy+a_3y=x^3+a_2x^2+a_4x+a_6$ for some $a_1,\ldots,a_6\in K$ and then let $\displaystyle \omega_E=\frac{dx}{2y+a_1x+a_3}$ and $\displaystyle \omega_{\n E}=\frac{dx}{2y+\n a_1x+\n a_3}$. Then $\mu^*(\omega_{\n E})=\alpha\cdot\omega_E$, where $\alpha\in K$ has norm $m$.

		From now on, $\omega_E$ and $\omega_{\n E}$ will be invariant differentials on $E$ and $\n E$, respectively, such that $\mu^*(\omega_{\n E})=\alpha\cdot\omega_E$ where $\alpha\in K$ has norm $m$. Write $\alpha=p+q\sqrt{\Delta_K}$ for some $p,q\in \Q$ (note that $q\neq 0$ because by assumption $m$ is not a square in $\Q$). Then we get a splitting map $\beta\colon \gal(F/\Q)\to F^*$ by setting $\displaystyle \beta(\sigma)=\frac{p}{q}+\frac{1}{q}\sqrt{m}$, since then $N_{F/\Q}(\beta(\sigma))=\Delta_K$. The splitting map $\beta$ that we get in this way induces an endomorphism $w$ of the abelian variety $A_f$ as described in Theorem \ref{betaend}, and the image $w(A_f)\coloneqq B$ is isogenous to $E$ since both $B$ and $E$ are building blocks of $A_f$. Let $\varphi\colon B\to E$ be an isogeny of minimal degree: with a little abuse of notation we will denote the composition $\varphi\circ w$ by $w$. From now on we set $\pi\coloneqq w\circ j_1 \colon X_1(N)\to E$, where $j_1$ is the natural map $X_1(N)\to J_1(N)\to A_f$. The pullback $\pi^*(\omega_E)$ lies in the $K$-vector space spanned by the form $h$ corresponding to $\beta$, so we have $\pi^*(\omega_E)=\gamma\cdot h$ for some $\gamma\in K^*$.
			\begin{lemma}\label{gammarational}
				We have $\displaystyle \frac{\n\gamma}{\gamma}=\pm 1$ and therefore $\gamma^2=\pm N_{K/\Q}(\gamma)$.
		\end{lemma}
			\begin{proof}
				Let $\omega_E'=\omega_E/\gamma$ and $\omega_{\n E}'=\omega_{\n E}/\n\gamma$, so that $\pi^*(\omega_E')=h$ and $\n\pi^*(\omega_{\n E}')=\n h$. As noticed in \cite[p. 488]{gola}, there exists an isogeny $\mu_{\nu}\colon E\to \n E$ induced by the action of some Hecke operator $T_p$ on $J_1(N)$ for an inert prime $p$ such that $\mu_{\nu}^*(\omega_{\n E}')=\n\lambda_p\cdot\omega_E'$ where $\lambda_p$ is the $p$-th coefficient of $h$. This implies $\mu_{\nu}^*(\omega_{\n E})=\n\lambda_p\cdot\frac{\n\gamma}{\gamma}\cdot\omega_E$. On the other hand, $\Hom(E,\n E)\otimes \Q\simeq \Q$ and this means that there exists some $s\in \Q^*$ such that $\mu_{\nu}=s\cdot \mu$ and thus
				$$\mu^*(\omega_{\n E})=\frac{\n\lambda_p}{s}\cdot\frac{\n\gamma}{\gamma}\cdot\omega_E.$$
				This implies that $N_{K/\Q}\left(\frac{\n\lambda_p}{s}\right)=m$. Now recall that $\displaystyle \lambda_p=\frac{a_p+\kappa\s a_p}{1+\kappa}$. Since $p$ is inert, the computations of section \ref{newform} show that $a_p$ equals $t\sqrt{m}$ for some $t\in \Q$. Thus we have
				$$\lambda_p=t\sqrt{m}\cdot \frac{1-\kappa}{1+\kappa},$$
				which implies $t=\pm s$ because $N_{K/\Q}(\lambda_p)=\eta s^2 m$ and $N_{K/\Q}\left(\frac{1-\kappa}{1+\kappa}\right)=\eta$. But now
				$$\frac{\lambda_p}{s}=\pm\sqrt{m}\cdot\frac{1-\frac{\sqrt{\Delta_K}}{p/q+1/q\sqrt{m}}}{1+\frac{\sqrt{\Delta_K}}{p/q+1/q\sqrt{m}}}=\pm\sqrt{m}\cdot\frac{\n\alpha+\sqrt{m}}{\alpha+\sqrt{m}}=\pm \n\alpha$$
				and therefore $\displaystyle \frac{\n\lambda_p}{s}=\pm \alpha$, showing that $\displaystyle \frac{\n\gamma}{\gamma}=\pm 1$ because $\mu^*(\omega_{\n E})=\alpha\cdot\omega_E$.
		\end{proof}
		\subsection{The images $0$ and $i\infty$ on $E$}

		The goal of this section is to prove that the points $\pi(0),\pi(i\infty)\in E(\OQ)$ are defined over $K$. If $\Gamma$ is a subgroup of the modular group $\SL_2(\Z)$ such that $\Gamma_1(N)\sub\Gamma\sub\Gamma_0(N)$, then the modular curve $X(\Gamma)$ has a model over $\Q$; however, its cusps may not be defined over $\Q$. In fact the following theorem holds.
			\begin{theorem}[Stevens, \cite{stev}]\label{actioncusps}
				Let $\Gamma$ be as above. Then:
					\begin{enumerate}
						\item[i)] the cusps of $X(\Gamma)$ are defined over $\Q(\xi_N)$ (where $\xi_N=e^{2\pi i/N}$);
						\item[ii)] for $d\in (\Z/N\Z)^*$ let $\tau_d$ be the element of $\gal(\Q(\xi_N)/\Q)$ satisfying $\leftidx{^{\tau_d}}\xi_N=\xi_N^d$ and let $\binom{x}{y}$ be a cusp of $\Gamma$. Then 
								$$\leftidx{^{\tau_d}}{\binom{x}{y}}=\binom{x}{d'y},$$
							  where $d'$ is a multiplicative inverse of $d$ modulo $N$.		
				\end{enumerate}
		\end{theorem}
		Let $\pi\colon X_1(N)\to E$ be our modular parametrization, which is defined over $K$. Recall that $\pi$ is the composition of the map $j_1\colon X_1(N)\to A_f$ and the map $w\colon A_f\to E$.
		The character $\chi$ can be viewed as a Dirichlet character modulo $N$ and if $H=\ker \chi$ then $\Q(\xi_N)^{H}=K$. Let us introduce the following congruence subgroup:
		$$\Gamma_H=\left\{\barr{cc}a & b\\c & d\narr\in \Gamma_0(N)\colon a,d\in H\right\}.$$
			\begin{lemma}\label{cuspdef}
				The points $\pi(0),\pi(i\infty)\in E(\OQ)$ are defined over $K$.
		\end{lemma}
			\begin{proof}
				First notice that $\Gamma_1(N)\sub \Gamma_H\sub\Gamma_0(N)$ and therefore we can apply Theorem \ref{actioncusps} to $X(\Gamma_H)$. The cusp $i\infty\in X(\Gamma_H)$, which is represented by $\binom{1}{0}$ is in fact defined over $\Q$ since for every $d\in (\Z/N\Z)^*$ one has ${}^{\tau_d}\binom{1}{0}=\binom{1}{0}$. On the other hand, ${}^{\tau_d}\binom{0}{1}=\binom{0}{d'}$. Thus if $d\in H$ also $d'\in H$ and the cusps $\binom{0}{1}$ and $\binom{0}{d'}$ are equivalent via any element of $\Gamma_H$ of the form $\barr{cc}a & Nb\\ Nc & d'\narr$ with $a,b,c\in \Z$. This shows that the cusp $0\in X(\Gamma_H)$ is defined over $K$.
				Now observe that $f\in S_2(\Gamma_0(N),\varepsilon)$ and $S_2(\Gamma_0(N),\varepsilon)\sub S_2(\Gamma_H)$ by definition of $\Gamma_H$ because either $\varepsilon$ is trivial or $\varepsilon=\chi$. This shows that $A_f$ is a quotient of the jacobian of $X(\Gamma_H)$. Thus the following diagram of varieties over $\Q$ commutes:
				$$\begin{xymatrix}{
				   X_1(N)\ar[r]^{p}\ar[d]_{j_1} & X(\Gamma_H) \ar[ld]^{j_2}\\
					A_f
			  }\end{xymatrix}$$
				where $p$ is the map induced by the inclusion $\Gamma_1(N)\sub\Gamma_H$ and $j_2$ is defined analogously to $j_1$, namely it is the composition $X(\Gamma_H)\to J(\Gamma_H)\to A_f$. Now the claim follows from the fact that $p(0)=0$ and $p(i\infty)=i\infty$.
		\end{proof}
		We shall now distinguish the two cases $\Delta_K>0$ and $\Delta_K<0$.

		\subsection{$K$ is real}

		First we recall the following lemma.
			\begin{lemma} \label{reallattice}
				Let $\Lambda\sub \C$ be a lattice with modular invariants $g_2(\Lambda),g_3(\Lambda)\in \R$. Then $\Lambda$ has a $\Z$-basis $\{\omega_1,\omega_2\}$ with $\omega_1\in \R_{>0}$ and $\Im (\omega_2)>0$. Moreover, such $\omega_1$ is unique.
		\end{lemma}
			\begin{proof}
				Note that for every lattice $\Lambda$ we have $\overline{g_i(\Lambda)}=g_i(\overline{\Lambda})$ for $i=2,3$ where the bar denotes complex conjugation. Since $g_2(\Lambda),g_3(\Lambda)\in \R$, this implies that $\Lambda=\overline{\Lambda}$. Therefore we can let $\displaystyle \omega_1\coloneqq\min\{|\omega|\colon\omega\in \Lambda\cap \R\setminus\{0\}\}$ (note that this set is always nonempty). If $\{x,y\}$ is a $\Z$-basis for $\Lambda$, then $\omega_1=ax+by$ for some $a,b\in \Z$ and $(a,b)$ must be $1$ by the minimality of $\omega_1$. Thus $\{\omega_1\}$ can be completed to a $\Z$-basis of $\Lambda$ and we have the claim.
		\end{proof}
		Let $\Lambda$ and $\Lambda_{\nu}$ be the period lattices of $(E,\omega_E)$ and $(\n E,\omega_{\n E})$, respectively. These can be identified with $\displaystyle\left\{\int_{\gamma}\omega_E\colon \gamma\in H_1(E,\Z)\right\}$ and $\displaystyle\left\{\int_{\gamma}\omega_{\n E}\colon \gamma\in H_1(\n E,\Z)\right\}$, respectively. Since $K$ is real, complex conjugation acts on $E(\C)$ and consequently induces involutions $\iota,\iota_{\nu}$ on $H_1(E,\Z)$ and $H_1(\n E,\Z)$, respectively. Therefore it is possible to find bases $\{\gamma_1,\gamma_2\}$ of $H_1(E,\Z)$ and $\{\gamma_{1,\nu},\gamma_{2,\nu}\}$ of $H_1(\n E,\Z)$ such that $\iota(\gamma_1)=\gamma_1$ and $\iota_{\nu}(\gamma_{1,\nu})=\gamma_{1,\nu}$. Let
		$$\omega_1\coloneqq \int_{\gamma_1}\omega_E\mbox{ and }\omega_{1,\nu}\coloneqq \int_{\gamma_{1,\nu}}\omega_{\n E}.$$
		Then $\omega_1,\omega_{1,\nu}$ are the unique positive real elements of $\Lambda$ and $\Lambda_{\nu}$ respectively which can be completed to bases as in Lemma \ref{reallattice}.
			\begin{definition}
				The positive real numbers $\omega_1$ and $\omega_{1,\nu}$ are called the \emph{real periods} of $(E,\omega_E)$.
		\end{definition}
		Since we have $\alpha\Lambda\sub \Lambda_{\nu}$, there exist $a,b\in \Z$ such that $\alpha\omega_1=a\omega_{1,\nu}+b\omega_{2,\nu}$. Since $\omega_1,\omega_{1,\nu},\alpha$ are real, $b$ must be equal to $0$. Therefore the following relation holds, for some non-zero integer $a$:
			\begin{equation}\label{realperiods}
				\omega_{1,\nu}=\frac{\alpha}{a}\omega_1.
		\end{equation}
		Note that $a$ divides $m$ because $\alpha\Lambda$ is a sublattice of $\Lambda_{\nu}$ of index $m$.
		We are now ready to compute with \eqref{Lvaluequadratic2}. Let $t\coloneqq |E(K)_{\text{tors}}|$. Then we have
			\begin{equation}\label{intermLvalue}
				t^2\cdot L(E/K,1)=\end{equation}
				$$=\frac{2+\kappa +\eta\s\kappa}{4}\cdot\frac{1}{\gamma^2}\cdot\left(t\cdot\int_{\{0,i\infty\}}\pi^*(\omega_E)\right)^2+\frac{2-\kappa -\eta\s\kappa}{4}\cdot\frac{1}{\n\gamma^2}\cdot\left(t\cdot\int_{\{0,i\infty\}}\n\pi^*(\omega_{\n E})\right)^2.$$
		Now use the fact that $\displaystyle t\cdot\int_{\{0,i\infty\}}\pi^*(\omega_E)=\int_{t\cdot \pi_*\{0,i\infty\}}\omega_E$. Then note that $\pi_*\{0,i\infty\}\in H_1(E,\Q)$ by Theorem \ref{mandri} and so by Lemma \ref{cuspdef} we have $\pi(0)-\pi(i\infty)\in E(K)_{\text{tors}}$. This implies that $t\cdot\pi_*\{0,i\infty\}\in H_1(E,\Z)$; moreover points on the imaginary axis in $\mathcal H$ are defined over $\R$ because complex conjugation on $X_1(N)$ corresponds to reflection with respect to the imaginary axis and the parametrization $\pi$ is defined over $\R$ once we have chosen an embedding $K\to \R$. Thanks to these remarks, we can say that $t\cdot\pi_*\{0,i\infty\}\in H_1(E,\Z)$ is invariant under complex conjugation and therefore there exists an integer $M$ such that
		$$\int_{t\cdot \pi_*\{0,i\infty\}}\omega_E=M\omega_1.$$
		The above argument can be repeated analogously for $\n\pi$, implying the existence of an integer $M'$ such that
		$$\int_{t\cdot \n\pi_*\{0,i\infty\}}\omega_{\n E}=M'\omega_{1,\nu}.$$
		Despite the fact that the argument used for $\n\pi$ is the same as the one used for $\pi$, the integers $M$ and $M'$ do not seem to be deeply related; in the example of level $229$ cited in \cite[p.~499]{gola} one has $M\neq 0$ while $M'=0$ (setting $f=f_2$ and $\s f=f_1$ in the notation of the paper). This is due to the fact that the eigenvalue of the Fricke involution $W_{229}$ applied to $f$ is exactly our $\kappa$, causing $\int_0^{i\infty}\n h(t)dt$, and consequently $M'$, to vanish. On the other hand one can check that $L(f,1)\cdot L(\s f,1)$ is non-zero, implying that $M\neq 0$.

		Finally, note that $\displaystyle \eta\s\kappa=\frac{\sqrt{\Delta_K}}{\s\beta}$ so that
		$$\frac{2+\kappa+\eta\s\kappa}{4}=\frac{2+\sqrt{\Delta_K}\left(\dfrac{\beta+\s\beta}{\Delta_K}\right)}{4}=\frac{\alpha}{2q\sqrt{\Delta_K}}$$
		and symmetrically
		$$\frac{2-\kappa-\eta\s\kappa}{4}=-\frac{\n\alpha}{2q\sqrt{\Delta_K}}.$$
		Substituting everything in \eqref{intermLvalue} and keeping \eqref{realperiods} and Lemma \ref{gammarational} in mind we get
			\begin{equation*}\begin{split}
				t^2\cdot L(E/K,1) & =\pm\frac{1}{2q\sqrt{\Delta_K}N(\gamma)}\left(\alpha\cdot M^2\cdot \omega_1^2-\n\alpha\cdot M'\cdot \omega_{1,\nu}^2\right)=\\
				& =\pm\frac{\omega_1\omega_{1,\nu}}{2q\sqrt{\Delta_K}N(\gamma)}\cdot \left(aM^2-\frac{m}{a}M'^2\right),
				\end{split}
		\end{equation*}
		which shows that $\displaystyle L(E/K,1)\cdot \frac{\sqrt{\Delta_K}}{\omega_1\omega_{1,\nu}}\in \Q$, and since $\displaystyle \frac{m}{a}\in \Z$ we get that
			\begin{equation}\label{realquadbound}
				L(E/K,1)=\frac{\omega_1\omega_{1,\nu}}{2q|E(K)_{\text{tors}}|^2\sqrt{\Delta_K}}\cdot \frac{w}{|N(\gamma)|} \mbox{ for some }w\in \Z.
		\end{equation}

		\subsection{$K$ is imaginary}

		The first observation in this case is that
		$$\frac{(1-\kappa)(1-\eta\s\kappa)}{4}L(\n h,1)^2=\overline{\frac{(1+\kappa)(1+\eta\s\kappa)}{4}L(h,1)^2},$$
		so that equation \eqref{Lvaluequadratic2} can be read as
				$$L(E/K,1)=2\Re\left(\frac{(1+\kappa)(1+\eta\s\kappa)}{4}L(h,1)^2\right).$$
		The same argument with $t\coloneqq |E(K)_{\text{tors}}|$ applies as in the case $\Delta_K>0$, so that we have $t\cdot\pi_*\{0,i\infty\}\in H_1(E,\Z)$. Let $\Lambda=\Z\omega_1+\Z\omega_2$ be the period lattice of $(E,\omega_E)$. We can assume that $\omega_{\n E}$ is such that $\overline{\Lambda}=\Z\overline{\omega_1}+\Z\overline{\omega_2}$ is the period lattice of $(\n E,\omega_{\n E})$. Thus there exist $a,b,c,d\in \Z$ such that
		$$\begin{cases}
			 \alpha\omega_1=a\overline{\omega_1}+b\overline{\omega_2} & \\
			\alpha\omega_2=c\overline{\omega_1}+d\overline{\omega_2}. 
		\end{cases}$$
		This time we have
		$$\int_{t\cdot \pi_*\{0,i\infty\}}\omega_E=(x\omega_1+y\omega_2)$$
		for some $x,y\in \Z$ and thus we get
			\begin{equation*}\begin{split}
				t^2\cdot L(E/K,1)& =\Re\left(\pm\frac{\alpha}{2q\sqrt{\Delta_K}}\cdot\frac{1}{N(\gamma)}\cdot(x\omega_1+y\omega_2)^2\right)=\\
				& =\pm\frac{2}{2q\sqrt{|\Delta_K|}N(\gamma)}\Im((x\omega_1+y\omega_2)(x(a\overline{\omega_1}+b\overline{\omega_2})+y(c\overline{\omega_1}+d\overline{\omega_2})))=\\
				& =\pm \frac{2\Im(\omega_1\overline{\omega_2})}{2q\sqrt{|\Delta_K|}N(\gamma)}\cdot (xy(a-d)+y^2c-x^2b),
			\end{split}
		\end{equation*}
		where we used the fact that $\sqrt{\Delta_K}$ is purely imaginary. Since $xy(a-d)+y^2c-x^2b\in \Z$ we get $\displaystyle L(E/K,1)\cdot\frac{\sqrt{|\Delta_K|}}{2\Im(\omega_1\overline{\omega_2})}\in \Q$ and therefore
			\begin{equation}\label{imgquadbound}
				L(E/K,1)=\frac{2\Im(\omega_1\overline{\omega_2})}{2q|E(K)_{\text{tors}}|^2\sqrt{|\Delta_K|}}\cdot \frac{w}{N(\gamma)}\mbox{ for some } w\in \Z.
		\end{equation}
		The term $\Im(\omega_1\overline{\omega_2})$ coincides (in absolute value) with the covolume of $\Lambda$.
		\section{The Manin ideal}\label{manid}
		The factor $N(\gamma)$ in \eqref{realquadbound} and \eqref{imgquadbound} should play a similar role to the one played by the Manin constant $c$ of \eqref{stima}. Let $\mathcal E$ be the N\'{e}ron model of $E$ over $\mathcal O_K$. Then $H^0(\mathcal E,\Omega_{\mathcal E/\mathcal O_K}^1)$ is a locally free $\mathcal O_K$-module of rank $1$ inside $H^0(E,\Omega_{E/K}^1)$. In \cite{gola}, the authors introduce the following fractional ideal attached to a parametrization $\pi\colon X_1(N)\to E$ over $K$.
			\begin{definition}
				The \emph{Manin ideal} $\mathfrak c(\pi)$ attached to the parametrization $\pi$ is the fractional ideal of $K$ satisfying:
				$$\pi^*H^0(\mathcal E,\Omega_{\mathcal E/\mathcal O_K}^1)=\mathfrak c(\pi)\left(\pi^*H^0(E,\Omega_{E/K}^1)\cap \mathcal O_K\llbracket q\rrbracket\right).$$
		\end{definition}
		If $\omega\in H^0(E,\Omega_{E/K}^1)$ is non-zero, let
		$$\mathfrak m_{\omega}(\pi)=\{x\in K\colon x\cdot \pi^*(\omega)\in \mathcal O_K\llbracket q\rrbracket dq\}.$$
		Following \cite{gola} again, we define the \emph{Weierstrass ideal} attached to the pair $(E,\omega)$ as the fractional ideal of $K$ defined by
				$$\delta_{\omega}=\prod_{\p} \p^{\ord_{\p}(\omega/\omega_{\p})},$$
		where $\p$ varies among all prime ideals of $\mathcal O_K$ and $\omega_{\p}$ is a minimal differential at $\p$.
			\begin{lemma}[\cite{gola}]
				For any non-zero $\omega\in H^0(E,\Omega_{E/K}^1)$ we have
					$$\mathfrak c(\pi)=(\mathfrak m_{\omega}(\pi)\delta_{\omega})^{-1}.$$
		\end{lemma}
		In analogy with Theorem \ref{manconisint}, the following holds.
			\begin{theorem}[\cite{gola}]
				The Manin ideal $\mathfrak c(\pi)$ is an integral ideal.
		\end{theorem}
		The set of pairs $(E,\pi)$, where $E$ is a $\Q$-curve completely defined over $K$ and $\pi\colon X_1(N)\to E$ is a modular parametrization over $K$, can be given the same ordering we used in section \ref{curvesoverQ} for parametrizations over $\Q$: given two parametrizations $(E,\pi)$ and $(E',\pi')$ we say that $(E',\pi')$ \emph{dominates} $(E,\pi)$, and we write $(E',\pi')\geq (E,\pi)$, if there exists an isogeny $\varphi\colon E'\to E$ such that $\pi=\varphi\circ\pi'$. A maximal element with respect to this ordering is called an \emph{optimal parametrization}, and it can be shown that every parametrization factors through an optimal one.

		Conjecture \ref{mancon} is therefore generalized as follows in \cite{gola}.
			\begin{conjecture}[Generalized Manin conjecture]\label{genmancon}
				Let $\pi\colon X_1(N)\to E$ be an optimal parametrization. Then $\mathfrak c(\pi)=(1)$.
		\end{conjecture}
		
		\subsection{The term $N_{K/\Q}(\gamma)$}
		Let us come back to our parametrization $\pi\colon X_1(N)\to E$ defined over $K$. Recall that $\pi^*(\omega_E)=\gamma\cdot h$. Now let $\pi'\colon X_1(N)\to E'$ be an optimal parametrization and $\psi\colon E'\to E$ a $K$-isogeny such that $\pi=\psi\circ \pi'$. Note that, as we did in section \ref{curvesoverQ}, we can assume that $\psi$ is an isogeny of minimal degree in $\Hom(E',E)$. In fact, let $\varphi$ be an element of minimal degree in $\Hom(E',E)$. Then there exists some integer $k$ such that $\psi=k\varphi$. Let $\overline{\pi}\coloneqq \varphi\circ\pi'$. Since $\pi=[k]\circ\overline{\pi}$, where $[k]$ denotes multiplication by $k$, we have $\displaystyle \overline{\pi}^*(\omega_E)=\frac{\gamma}{k}h$. Thus we could replace $\pi$ by $[k]\circ\overline{\pi}$ in our computations which led to \eqref{realquadbound} and \eqref{imgquadbound}, and the only effect in these estimates would be to replace $\gamma$ by $\gamma/k$, which multiplies the value of $L(E/K,1)$ by $k^2$. Therefore we can assume $\psi=\varphi$.
		The next step is to understand how $\mathfrak c(\pi)$ and $\mathfrak c(\pi')$ are related. Note that
		$$\mathfrak m_{\psi^*(\omega_E)}(\pi')=\{x\in K\colon x\cdot \pi'^*(\varphi^*(\omega_E))\in \mathcal O_K\llbracket q\rrbracket\}=\mathfrak m_{\omega_E}(\pi),$$
		so that
		$$\mathfrak c(\pi)=\mathfrak c(\pi')\delta_{\psi^*(\omega_E)}\delta_{\omega_E}^{-1}.$$
		If we set $\widetilde{\omega}_E\coloneqq \frac{1}{\gamma}\omega_E$, then we have $\pi^*(\widetilde{\omega}_E)=h$ and therefore $\mathfrak m_{\widetilde{\omega}_E}(\pi)$ coincides with the denominator ideal of $h$, i.e. the ideal
		$$D_h=\{x\in K\colon x\cdot h\in \mathcal O_K\llbracket q\rrbracket\}.$$
		Note that this is an integral ideal because $h$ is normalized. Thus we have
			\begin{equation}\label{denomideal}
				N_{K/\Q}(D_h)=N_{K/\Q}(\mathfrak m_{\widetilde{\omega}_E}(\pi))=N_{K/\Q}(\mathfrak c(\pi')^{-1}\delta_{\psi^*(\widetilde{\omega}_E)}^{-1})=\frac{1}{N_{K/\Q}(\delta_{\psi^*(\widetilde{\omega}_E)})}
		\end{equation}
		under conjecture \ref{genmancon}. It is easy to see that
		$$\delta_{\psi^*(\widetilde{\omega}_E)}=\prod_{\p}\p^{\ord_{\p}(\psi^*(\widetilde{\omega}_E)/\omega_{\p}')}=\prod_{\p}\p^{-\ord_{\p}(\gamma)}\cdot \prod_p\p^{\ord_{\p}(\psi^*(\omega_E)/\omega_{\p}')},$$
		where $\omega_{\p}'$ is a minimal differential at $\p$ on $E'$, and thus
		$$N_{K/\Q}(\delta_{\psi^*(\widetilde{\omega}_E)})=\frac{N_{K/\Q}(\delta_{\psi^*(\omega_E)})}{N_{K/\Q}(\gamma)},$$
		where $N_{K/\Q}(\gamma)$ is the norm of the fractional ideal generated by $\gamma$, which coincides with the absolute value of the norm of the element $\gamma$.
		By \eqref{denomideal} we get
			\begin{equation} \label{normofgamma1}
				\frac{1}{N_{K/\Q}(\gamma)}=\frac{1}{N_{K/\Q}(D_h)N_{K/\Q}(\delta_{\psi^*(\omega_E)})}.
		\end{equation}
		Next we claim that there exists an integer $v>0$ such that
			\begin{equation} \label{normofdelta}
				N_{K/\Q}(\delta_{\psi^*(\omega_E)})\cdot v= N_{K/\Q}(\delta_{\omega_E})N_{K/\Q}(\deg\psi).
		\end{equation}
		Let $\omega'$ be a differential on $E'$ and let $a,b\in K$ be such that $\psi^*(\omega_E)=a\omega'$ and $\widehat{\psi}^*(\omega')=b\omega_E$ where $\widehat{\psi}$ is the dual isogeny, so that $ab=\deg\psi$. For each prime $\p$ of $K$ let $a_{\p},b_{\p}\in K$ be such that $\omega_E=a_{\p}\omega_{\p}$ and $\omega'=b_{\p}\omega_{\p}'$. With these notations we have
		$$\widehat{\psi}^*(\omega_{\p}')=\frac{ba_{\p}}{b_{\p}}\omega_{\p}.$$
		By the functoriality of the N\'{e}ron model, the pullback of an integral differential is integral. Thus we have
			\begin{equation}\label{due}
				\ord_{\p}\left(\frac{ba_{\p}}{b_{\p}}\right)\geq 0 \mbox{ for all } \p.
		\end{equation}
		Now
		$$\delta_{\psi^*(\omega_E)}=\prod_{\p}\p^{\ord_{\p}(\psi^*(\omega_E)/\omega_{\p}')}=\prod_{\p}\p^{\ord_{\p}(ab_{\p}\omega_{\p}'/\omega_{\p}'))}=\prod_{\p}\p^{\ord_{\p}(ab_{\p})},$$
		while
		$$\delta_{\omega_E}=\prod_{\p}\p^{\ord_{\p}(\omega_E/\omega_{\p})}=\prod_{\p}\p^{\ord_{\p}(a_{\p})}.$$
		Therefore the claim \eqref{normofdelta} is proved if for all $\p$ we have
		$$\ord_{\p}(ab_{\p})\leq\ord_{\p}(a_{\p})+\ord_{\p}(\deg\psi).$$
		In fact we can rewrite this condition, using the fact that $\ord_{\p}(a)+\ord_{\p}(b)=\ord_{\p}(\deg\psi)$, as
		$$\ord_{\p}(b)+\ord_{\p}(a_{\p})-\ord_{\p}(b_{\p})\geq 0,$$
		which is exactly \eqref{due}. Finally equation \eqref{normofdelta} together with \eqref{normofgamma1} implies that
			\begin{equation}\label{normofgamma}
				\frac{1}{N_{K/\Q}(\gamma)}=\frac{v}{N_{K/\Q}(D_h)N_{K/\Q}(\delta_{\omega_E})N_{K/\Q}(\deg\psi)} \mbox{ for some } v\in \Z_{>0}.
		\end{equation}

		\section{Completing the proof}\label{finbnd}
		The last two things we are left to understand are the norm of the denominator ideal $D_h$ and the degree of the isogeny $\psi$.

		\subsection{The denominator ideal $D_h$}

		We will now compute the denominator ideal $D_h$, which is integral as we already noticed. Recall that
		$$h=\frac{1}{1+\kappa}f+\frac{\kappa}{1+\kappa}\s f,$$
		where $\displaystyle \kappa=\frac{\sqrt{\Delta_K}}{\beta}$. Let $\displaystyle f=\sum_{n=1}^{+\infty}a_nq^n$ and $\displaystyle h=\sum_{n=1}^{+\infty}\lambda_nq^n$. Then for every $n\geq 1$ we have
		$$\lambda_n=\frac{1}{1+\kappa}a_n+\frac{\kappa}{1+\kappa}\s a_n.$$
		Recall that since $f$ is a normalized newform, the $a_n$'s are algebraic integers, so they belong to $\mathcal O_F$. If $m\equiv 2,3\bmod 4$ then every $a_n$ is of the form $a+b\sqrt{m}$ for some $a,b\in \Z$. Thus we have
		\begin{equation*}\begin{split}
		\lambda_n-a & =\frac{1-\kappa}{1+\kappa}\cdot b\sqrt{m}=\frac{\sqrt{\Delta_K}-\s\beta}{\sqrt{\Delta_K}+\s\beta}\cdot b\sqrt{m}=\frac{q\sqrt{\Delta_K}-p+\sqrt{m}}{q\sqrt{\Delta_K}+p-\sqrt{m}}\cdot b\sqrt{m}\\
		 & =\frac{-\n\alpha+\sqrt{m}}{\alpha+\sqrt{m}}\cdot b\sqrt{m}=\frac{(-\n\alpha+\sqrt{m})(\alpha+\sqrt{m})}{\alpha^2-m}\cdot b\sqrt{m}\\
			& =\frac{(\alpha-\n\alpha)\sqrt{m}}{\alpha^2-m}\cdot b\sqrt{m}=\frac{bm}{\alpha}=b\n\alpha,
		\end{split}
		\end{equation*}
		using the fact that $\alpha\n\alpha=m$.
		If $m\equiv 1\bmod 4$ and $\displaystyle a_n=a+b\left(\frac{1+\sqrt{m}}{2}\right)$ for some $a,b\in \Z$, one sees in the same way that
		$$\lambda_n=a+\frac{b}{2}+\frac{b\cdot\n\alpha}{2}.$$
		Let $D_{\n\alpha}=\{x\in \mathcal O_K\colon x\cdot \n\alpha\in \mathcal O_K\}$ be the denominator ideal of $\n\alpha$, which clearly coincides with $\n D_{\alpha}$. Then what we have shown is that if $m\equiv 2,3\bmod 4$ then $\n D_{\alpha}\sub D_h$, while if $m\equiv 1\bmod 4$ then $2 \cdot \n D_{\alpha}\sub D_h$. Since $N_{K/\Q}(D_{\alpha})=N_{K/\Q}(\n D_{\alpha})$, this gives us the following useful lemma.
			\begin{lemma}\label{denomid}
				Let $D_h$ be the denominator of $h$. Then we have the following two cases:
				$$\begin{dcases}
					 \mbox{if } m\equiv 2,3\bmod 4 & \mbox{then } N_{K/\Q}(D_h)\mid N_{K/\Q}(D_{\alpha})\\
					\mbox{if } m\equiv 1\bmod 4 & \mbox{then } N_{K/\Q}(D_h)\mid 4N_{K/\Q}(D_{\alpha}).
				\end{dcases}$$
		If in particular $\alpha\in \mathcal O_K$, then:
				$$\begin{cases}
					 \mbox{if } m\equiv 2,3\bmod 4 & \mbox{then } N_{K/\Q}(D_h)=1\\
					 \mbox{if } m\equiv 1\bmod 4 & \mbox{then } N_{K/\Q}(D_h)\in\{1,2,4\}.
				\end{cases}$$
		\end{lemma}

		\subsection{The isogeny $\psi$}
                \label{isogenygraph}

		The factor $N_{K/\Q}(\deg\psi)={(\deg\psi)}^2$ can be treated exactly in the same way as in the case of curves over $\Q$. Recall that $\psi$ is an isogeny of minimal degree in $\Hom_{K}(E',E)$. Since we might not be able to find the curve $E'$, we bound $\deg\psi$ in the following way. Let $\{E_1,\ldots,E_n\}$ be the $K$-isogeny class of $E$. For each $i=1,\ldots,n$ let $\displaystyle s_i\coloneqq \min_{\varphi}\{\deg\varphi\colon\varphi\in\Hom_K(E_i,E)\}$ and $s\coloneqq\gcd(s_i\colon i=1,\ldots,n)$. Then clearly $\deg\psi$ divides $s$. Finding the value of $s$ can be done algorithmically as illustrated in \cite{bil}. The author provides an algorithm which allows us, given an elliptic curve $C$ over a number field $K$, to compute the finite set of rational primes $\{p_1,\ldots,p_r\}$ such that $C$ admits a $K$-isogeny of degree $p_i$ for every $i$. Repeating this procedure a finite number of times allows us to draw a graph called the \emph{isogeny graph} whose vertices correspond to $\{E_1,\ldots,E_ n\}$ and such that for $i\neq j$ there is an edge from $E_i$ to $E_j$ if and only if there is an isogeny of prime degree between $E_i$ and $E_j$. This is a (weighted, undirected) connected graph because every isogeny can be decomposed as a chain of isogenies of prime degree. 
		\begin{remark}
			Assume that $E$ and $\n E$ are not isomorphic. Since being a $\Q$-curve is an invariant condition under isogeny and by assumption no curve in the isogeny class of $E$ is defined over $\Q$, the isogeny graph of $E$ has an even number of vertices, call them $\{E_1,\ldots,E_{2n}\}$. These can be labeled in the following way: we assume $E=E_1$ and for every $i=1,\ldots,n$ we set $E_{n+i}=\n E_i$. Then in order to find $s$ it is enough to consider the subgraph $\{E_1,\ldots,E_n\}$ because if any curve $E_{n+i}\in\{E_{n+1},\ldots,E_{2n}\}$ admits an optimal parametrization, then so does $E_i$: it is enough to consider the conjugate parametrization.
		\end{remark}

		\section{The main theorem}\label{mainsec}
		Let us now collect all the ingredients we have in order to be able to state our result in a more compact way.

		Let $K$ be a quadratic number field of discriminant $\Delta_K$ with Galois group $\{1,\nu\}$. Let $E/K$ be a $\Q$-curve with no CM, completely defined over $K$ and not isogenous to an elliptic curve defined over $\Q$. Let $\mu\colon E\to \n E$ be an isogeny and let $m$ be the integer such that $\n\mu \mu$ coincides with multiplication by $m$. Let $\omega_E$ be a invariant differential on $E$ and let $\omega_{\n E}$ be an invariant differential on $\n E$ such that $\mu^*(\omega_{\n E})=\alpha\cdot\omega_E$, where $\alpha=p+q\sqrt{\Delta_K}\in K$ has norm $m$. Let $D_{\alpha}=\{x\in \mathcal O_K\colon x\cdot \alpha\in \mathcal O_K\}$ be the denominator ideal of $\alpha$. Let $\delta_{\omega_E}$ be the Weierstrass ideal of $(E,\omega_E)$.
		
		Let $\omega_1,\omega_{1,\nu}$ be the (positive) real periods of $(E,\omega_E)$ if $K$ is real and let $\{\omega_1,\omega_2\}$ be a basis for the period lattice of $(E,\omega_E)$ such that $\Im(\omega_1\overline{\omega_2})>0$ if $K$ is imaginary. Define
		$$\Omega_E\coloneqq\begin{dcases}\frac{\omega_1\cdot\omega_{1,\nu}}{N_{K/\Q}(\delta_{\omega_E})} & \mbox{if $K$ is real}\\ \frac{2\Im(\omega_1\overline{\omega_2})}{N_{K/\Q}(\delta_{\omega_E})} & \mbox{if $K$ is imaginary.} \end{dcases}$$
		Notice that the product formula implies that $\Omega_E$ does not depend on $\omega_E$.
		\begin{remark}
			If $\omega_E$ is a global minimal differential on $E$, one has that $\delta_{\omega_E}=(1)$. This justifies the fact that in section \ref{curvesoverQ} we omitted the term coming from $\delta_{\omega_E}$ in the definition of $\Omega_E$ for elliptic curves over $\Q$. The two definitions are therefore consistent.
		\end{remark}

		Finally, let $s$ be the positive integer determined in section \ref{isogenygraph}. 
			\begin{theorem}\label{mainthm}
				Let the notation be as above. Then the following hold:
				\begin{enumerate}
				\item[i)] If $L(E/K,1)\neq 0$, then $\displaystyle L(E/K,1)\cdot\frac{\sqrt{|\Delta_K|}}{\Omega_E}\in \Q^*$.
				\item[ii)] Assume conjecture \ref{genmancon}; let $\mu^*(\omega_{\n E})=(p+q\sqrt{\Delta_K})\omega_E$ with $p,q\in \Q$, let $s$ be the $\text{lcm}$ of the minimal isogeny degrees between curves in the isogeny class of $E$ and let $t=4$ if $m\equiv 1 \bmod 4$ and $t=1$ otherwise. Then:
				$$L(E/K,1)\cdot \frac{\sqrt{|\Delta_K|}}{\Omega_E}\cdot 2q\cdot |E(K)_{\text{tors}}|^2\cdot  t\cdot N_{K/\Q}(D_{\alpha})\cdot s^2\in \Z.$$
				\end{enumerate}
		\end{theorem}
		Theorem \ref{mainthm} is the analogue of equation \ref{stima} that we were seeking for. It has the same type of applications of that equation: we can use it in order to compute the $L$-ratio $\displaystyle L(E,1)\cdot\frac{\sqrt{|\Delta_K|}}{\Omega_E}$ whenever such value is non-zero or to prove that $L(E,1)=0$ if this is the case. For this second application we can, as in section \ref{curvesoverQ}, substitute $s$ with $\displaystyle s'\coloneqq \max_i\{s_i\}$, in order to get a more efficient lower bound.
		\subsection{The Birch and Swinnerton-Dyer conjecture}\label{BSD}

		Let us now recall the statement of the Birch and Swinnerton-Dyer conjecture for elliptic curves over number fields. For a reference on the subject, see \cite{dok} or \cite{gro2}.
		For an elliptic curve $E$ over a number field $K$, we recall that the \emph{algebraic rank} of $E$ is the rank of $E(K)/E(K)_{\text{tors}}$ as a $\Z$-module, while the \emph{analytic rank} of $E$ is the order of vanishing of $L(E,s)$ at the point $s=1$. The analytic rank is only defined if $L(E,s)$ has an analytic continuation to $\C$ (or to any neighborhood of $s=1$), which is not known to be true in general. Therefore we will include the statement inside the conjecture.
			\begin{conjecture}[Weak BSD conjecture]
				 Let $E$ be an elliptic curve over a number field $K$. Then: 
					\begin{enumerate}
						\item[a)] $L(E,s)$ has an analytic continuation to $\C$;
						\item[b)] the analytic rank and the algebraic rank of $E$ coincide.
				\end{enumerate}
		\end{conjecture}
		Before stating the strong form the BSD conjecture, let us recall the definition of the following invariants attached to $E$.
		\begin{itemize}
		 \item Let $\{P_1,\dots,P_r\}$ be a $\Z$-basis of $E(K)/E(K)_{\text{tors}}$. The \emph{regulator} of $E$ over $K$, denoted by $R(E/K)$, is defined as
 		$$R(E/K)=\det(\langle P_i,P_j\rangle),$$
 		where $\langle , \rangle$ is the N\'{e}ron--Tate pairing of $E$ over $K$.
		\item Let $M_K=M_K^{\infty}\cup M_K^0$ be the set of places of $K$, where $M_K^{\infty}$ is the set of archimedean places and $M_K^0$ is the set of non-archimedean places. Choose an invariant differential $\omega$ on $E$. For every place $v\in M_K$, the invariant differential $\omega$ gives an invariant differential $\omega_v$ on $E_{K_v}$, where $K_v$ is the completion of $K$ at $v$. Let $dx$ be the Haar measure on the ring of ad\`{e}les $\mathbb A$ such that $\int_{\mathbb A/K}dx=1$ and choose a decomposition $dx=\otimes_vdx_v$, so that $dx_v$ is a Haar measure on $K_v$. Finally, for every $v \in M^0(K)$ let $L_v(E,s)$ be the local $L$-function at $v$ (see \cite{gro2} for details).
		
		The \emph{period} of $E$ over $K$ is defined as:
			$$P(E/K)=\prod_{v\in M^0_K}\left(L_v(E,1)\cdot\int_{E(K_v)}|\omega_v|\right)\cdot \prod_{v\in M^{\infty}_K}\int_{E(K_v)}|\omega_v|.$$
		By \cite[Lemma 54]{tat}, the product defining $P(E/K)$ is finite, since almost all factors are equal to $1$. The product formula shows that $P(E/K)$ is independent of $\omega$.
		\item The \emph{Tate--Shafarevich group} of $E$ over $K$ is defined as
		 $$\Sh(E/K)=\ker\left(H^1(G_K,E)\stackrel{\res}{\longrightarrow}\prod_{v\in M_K}H^1(G_{K_v},E)\right),$$
		where $G_K$ (resp. $G_{K_v}$) is the absolute Galois group of $K$ (resp. $K_v$) and 
		$$\res=\prod_{v\in M_K}\left(\res_v\colon H^1(G_K,E)\to H^1(G_{K_v},E)\right).$$
		\end{itemize}

			\begin{conjecture}[Strong BSD conjecture]
				The weak BSD conjecture holds and moreover we have that:
					\begin{enumerate}
						\item[c)] the Tate--Shafarevich group is finite and if $r$ is the rank of $E$ then:
							$$\frac{L^{(r)}(E,1)}{r!}=\frac{P(E/K)\cdot R(E/K)\cdot |\Sh(E/K)|}{|E(K)_{\text{tors}}|^2}.$$
				\end{enumerate}
		\end{conjecture}
		Let us explain the relationship between $P(E/K)$ and the quantity $\Omega_E/\sqrt{|\Delta_K|}$ appearing in Theorem \ref{mainthm}. Recall that if $v\in M^0_K$ and $\mathcal E$ is a minimal model of $E$ at $v$, then the \emph{Tamagawa number} of $E$ at $v$ is defined as $[\mathcal E(K_v)\colon \mathcal E^0(K_v)]$ where $\mathcal E^0(K_v)$ is the subgroup of $\mathcal E(K_v)$ consisting of points which reduce to nonsingular points modulo $v$. Note that there are only finitely many $v$ such that $c_v\neq 1$. In \cite[pp. 92-96]{lang} it is proved that if $\omega$ is an invariant differential on $E$ then
		\begin{equation}\label{period1}
		P(E/K)=\prod_{v\in M^0_K}c_v\cdot \prod_{\substack{v\in M_K^{\infty}\\v\mbox{ \tiny real}} }\int_{E(K_v)}|\omega|\cdot\prod_{\substack{v\in M_K^{\infty}\\v\mbox{ \tiny cplx}} }2\int_{E(K_v)}|\omega\wedge\overline{\omega}|\cdot\frac{1}{N(\delta_{\omega})\cdot\sqrt{|\Delta_K|}}.
		\end{equation}
 		Let $v$ be a real place of $K$ and let $\Lambda_v$ be the period lattice of $(E_{K_v},\omega_v)$. By Lemma \ref{reallattice}, $\Lambda_v$ has a basis of the form $\{\omega_{1,v},\omega_{2,v}\}$ where $\omega_{1,v}\in \R_{>0}$. Then the quantity $\displaystyle \int_{E(K_v)}|\omega|$ coincides with $[E_{K_v}(K_v)\colon E^0_{K_v}(K_v)]\cdot \omega_{1,v}$, where $E^0(K_v)$ is the connected component of $E_{K_v}$ containing the identity. Therefore $[E_{K_v}(K_v)\colon E^0_{K_v}(K_v)]$ is $2$ precisely when the whole $2$-torsion subgroup of $E_{K_v}$ is defined over $K_v$, and $1$ otherwise.

		When $v$ is a complex place of $v$ and $\Lambda_v$ is the period lattice of $(E_{K_v},\omega_v)$, then the term $2\int_{E(K_v)}\omega\wedge\overline{\omega}$ coincides with twice the covolume of $\Lambda_v$.
		
		Therefore equation \eqref{period1} gives:
		$$P(E/K)=\prod_{v\in M^0_K}c_v\cdot \prod_{\substack{v\in M_K^{\infty}\\v\mbox{ \tiny real}} }[E_{K_v}(K_v)\colon E^0_{K_v}(K_v)]\cdot \frac{\Omega_E}{\sqrt{|\Delta_K|}}.$$
		Recall that part i) of Theorem \ref{mainthm} tells us that if $L(E/K,1)\neq 0$ then $\displaystyle L(E/K,1)\cdot\frac{\sqrt{|\Delta_K|}}{\Omega_E}\in \Q^*$. Therefore our result is at least consistent with the statement of the BSD conjecture when $E$ has analytic rank $0$, since it shows that the ``irrational part'' of $L(E/K,1)$ is $\Omega_E/\sqrt{\Delta_K}$.

		\section{Examples}\label{exs}
		In this concluding section, we will provide explicit examples of quadratic $\Q$-curves, showing how it is possible to use Theorem \ref{mainthm} to verify that the analytic rank is positive or to compute the $L$-ratio, under Conjecture \ref{genmancon}. We remark that the newforms involved in examples $1,2,4$ and $6$ can be computed using computer software and existing algorithms (it took around $32$ minutes to compute the newform of example $2$, the one of largest level amongst the aforementioned). On the other hand, examples $3$ and $5$, which involve newforms of level whose order of magnitude is $10^8$ and $10^7$ respectively, cannot feasibly be treated with modular symbols methods, while they are easily handled using our algorithm.
		
		In order to find examples of $\Q$-curves, one can follow the method indicated in \cite{elk}. Let us briefly recall it. Let $N\in \N$ be square-free and consider the modular curve $X_0(N)$, whose $k$-rational points parametrize (isomorphism classes of) pairs $(E,\phi)$ where $E$ is an elliptic curve over a number field $k$ and $\phi$ is a degree $N$ isogeny with cyclic kernel defined over $k$. For every divisor $N_1$ of $N$ with $\gcd(N_1,N/N_1)=1$, there exists an involution $w_{N_1}$ on $X_0(N)$ which is defined as follows at non-cuspidal points: if $(E,\phi)\in Y_0(N)(k)$ and $N=N_1\cdot N_2$ then $\phi$ factors uniquely as $\phi_2\circ\phi_1$ where $\phi_i$ has degree $N_i$. Note that by the uniqueness of the factorization, the $\phi_i$'s are defined over $k$. On the other hand, $\phi$ factors as $\varphi_1\circ\varphi_2$ with $\varphi_i$ of degree $N_i$. If $\widehat{\phi_1}$ denotes the dual isogeny of $\phi_1$ then $\varphi_2\circ\widehat{\phi_1}$ is a cyclic $k$-rational isogeny of degree $N$ and it therefore corresponds to a point of $Y_0(N)$ which is $w_{N_1}((E,\phi))$. The set of all the $w_M$ for $M\mid N$ with $\gcd(M, N/M) = 1$ is an abelian group, denoted by $W(N)$, isomorphic to $(\Z/2\Z)^r$ where $r$ is the number of distinct prime factors of $N$. The quotient of $X_0(N)$ by $W(N)$ is denoted by $X^*(N)$; given a $\Q$-rational point $P$ on $X^*(N)$, its preimages on $X_0(N)$ under the quotient map $X_0(N)\to X^*(N)$ form a $G_{\Q}$-stable set whose elements correspond to $\Q$-curves. Conversely, in the same paper Elkies shows that every $\Q$-curve is geometrically isogenous to a $\Q$-curve that arises in this way.

		In \cite{has}, the author computes some families of $\Q$-curves defined over quadratic fields admitting an isogeny of small prime degree to the conjugates. Let us recall the equations of two such families (for more details see Theorem 2.2): for every square-free integer $d\neq 1$ and each rational number $u$ let
		$$E_{d,u}^{(2)}\colon y^2=x^3+6(3u\sqrt{d}-5)x-8(9u\sqrt{d}-7)$$
		$$E_{d,u}^{(7)}\colon y^2=x^3-Ax+B,$$
		where
		$$A=21(u^2d+27)(15u^2d+96u\sqrt{d}+85)$$
		$$B=98(u^2d+27)(27u^4d^2+144u^3d\sqrt{d}+1170u^2d+2608u\sqrt{d}+1539).$$
		Then $E_{d,u}^{(p)}$ is a $\Q$-curve admitting an isogeny of degree $p$ to its conjugate. In \cite[Corollary~4.3]{has} the author explains how to construct twists of the curves in this family which are completely defined over the base field. By searching through the families above twisted by some simple values $b$, we used these results to construct examples of $\Q$-curves of positive rank completely defined over quadratic fields. As noticed in \cite{bru}, the algebraic rank of such curves is necessarily even. This follows from the existence of an action of $\Z[\sqrt{m}]$, where $m=\pm p$, on $E(K)$. In particular, all curves in our examples have algebraic rank two: it can be checked using the algorithm in \cite{sim} that the rank is at most two; we will exhibit for each curve a pair of independent points of infinite order. In fact, a result announced by Tian and Zhang \cite[Theorem 4.3.2]{zhang} would prove that if $E$ is a quadratic $\Q$-curve completely defined over its base field and $E$ has analytic rank 2, then it has algebraic rank 2. On the other hand, if one is lucky enough to find a point of infinite order on $E$ and at the same time knows that the analytic rank is at most 2 (which can be checked numerically), then \cite[Corollary~14.3]{kato} shows that $L(E,1)=0$ and the functional equation \eqref{functionalequation} then implies that the analytic rank of $E$ is exactly 2. Our result, conditional on Conjecture~\ref{genmancon}, permits to prove, even without being able to find rational points, that the analytic rank is exactly 2 and thus, by Tian and Zhang's aforementioned result, that the algebraic rank is also 2.

		For every curve presented below we will compute the relevant invariants and the lower bound given  by Theorem \ref{mainthm}. Afterward, we will compute the (Galois orbit of the) newform attached to it and the corresponding sign of the functional equation. Finally, computing $L(E/K,1)$ and $L'(E/K,1)$ within a sufficient precision will allow us to verify the validity of the weak form of BSD conjecture for these curves. All computations have been performed using Sage \cite{sage}. The computations of $L(f,1)$ and $L'(f,1)$ rely on the algorithm presented in \cite{dokL}. This is based on the following well-known fact (see \cite{atkli} and \cite{dis}). Let $\displaystyle f=\sum_{n=1}^{+\infty}a_nq^n$ be a newform in $S_2(\Gamma_1(N))$. Then there exists an algebraic number $\eta_f\in \C^*$ of absolute value $1$ such that:
		\begin{equation}\label{functionalequation}
			\Lambda(f,s)=\eta_f\cdot\Lambda(f^*,2-s),
		\end{equation}
		where $\displaystyle f^*=\sum_{n=1}^{+\infty}\overline{a_n}q^n$ and $\Lambda(f,s)=N^{s/2}(2\pi)^{-s}\Gamma(s)L(f,s)$. The number $\eta_f$ is called the \emph{sign} of the functional equation \eqref{functionalequation} and it is $\pm 1$ when $F=\Q(a_n\colon n\in \N)$ is a totally real number field. The algorithm quoted above can be also used to compute $\eta_f$, when this is not known.
		
		Notice that whenever $f$ is the newform attached to a quadratic $\Q$-curve, $F$ is real and $\eta_f=-1$, then one can deduce trivially that $L(E/K,1)=0$.

		\subsubsection*{Example 1}
		Let $d=22$, $K=\Q(\sqrt{22})$, $u=-36/169$ and $\displaystyle b=\frac{91}{3}+\frac{13}{2}\sqrt{22}$. Then an integral model for the curve $E_{22,-36/169}^{(2)}$ twisted by $b$ is given by
		$$E\colon y^2 = x^3 - (7200684-1535112\sqrt{22})x +10456553952-2229344208\sqrt{22}.$$
                This is the curve with label 2.2.88.1-49.1-d4 in the LMFDB database \cite{lmfdb}.
		Note that since		$$j(E)=\frac{26692787554112}{2401}+\frac{5690844716544}{2401}\sqrt{22}$$
		is not integral, $E$ has no CM. There is a $K$-isogeny
		$$\mu\colon E\to \n E$$
		$$(x,y)\mapsto (g(x),y\cdot h(x)),$$
		where
		$$g(x)=\frac{(197/2+21\sqrt{22})x^2 -(1092+234\sqrt{22})x + 27378-5832\sqrt{22}}{x  + 2184-468\sqrt{22}}$$
		$$h(x)=\frac{-(2765/2+1179/4\sqrt{22})x^2 + (30732+6552\sqrt{22})x - 171162-36261\sqrt{22}}{x^2 + (4368-936\sqrt{22})x + 9588384-2044224\sqrt{22}},$$
		such that $\n\mu\mu=-2$. Therefore $F=\Q(\sqrt{-2})$. If $\displaystyle\omega_{E}=\frac{dx}{2y}$ and $\omega_{\n E}$ is its conjugate, we have that $\mu^*(\omega_{\n E})=(-14+\frac{3}{2}\sqrt{88})\omega_E$, so that we can set $\alpha=-14+\frac{3}{2}\sqrt{88}$ and consequently $\beta=-\frac{28}{3}+\frac{2}{3}\sqrt{-2}$. Clearly $N_{K/\Q}(\alpha)=-2$ and $N_{F/\Q}(\beta)=88$. Moreover, since $\alpha\in \mathcal O_K$ by Lemma \ref{denomid} it follows that $D_h=(1)$.

		One can check that $E$ has conductor $(7)$. The Weierstrass ideal of the standard invariant differential is $\delta_{\omega_E}=(6)^{-1}$.

		The last step is to find $s$ as in Theorem \ref{mainthm}. Using the algorithm described in \cite{bil}, one can check that the isogeny graph of $E$ has the following shape:
		$$\begin{xymatrix}{
				   E\ar@{-}[r]^{2}\ar@{-}[d]_{3} & \n E \ar@{-}[d]^{3}\\
					E_1 \ar@{-}[r]^2 & \n E_1.
			  }\end{xymatrix}$$
		Therefore either $E$ possesses an optimal parametrization or $E_1$ does, showing that we can assume $s=3$. The following table summarizes the invariants we need in order to apply Theorem \ref{mainthm}.
			\begin{center}
				\begin{tabular}{| c | c | c | c | c |}
			    \hline
				$\Omega_E$ & $q$ & $|E(K)_{\text{tors}}|$ & $t\cdot N_{K/\Q}(D_{\alpha})$ & $s^2$ \\ \hline\hline
			    $5.45882600014972$ & $3/2$ & $6$ & $1$ & $9$\\ \hline
			  \end{tabular}
		\end{center}
		By Theorem \ref{mainthm}, if $L(E/K,1)\neq 0$ then we must have that:
		$$|L(E/K,1)|\geq\frac{5.45882600014972}{3\cdot 36\cdot\sqrt{88}\cdot 9}\approx 5.98675727185567\cdot 10^{-4}.$$
		Two independent points of infinite order in $E(K)$ are given by 
		$$P=(-1860 +396 \sqrt{22}, 75924-16200 \sqrt{22})$$
		and
		$$Q=(498-72 \sqrt{22},-47628+10584\sqrt{22}).$$
		The newform $\displaystyle f=\sum_{n=1}^{+\infty}a_nq^n$ attached to $E$ has level $7\cdot 88=616$. This implies $f\in S_2(\Gamma_1(616),\varepsilon)$, where $\varepsilon$ is the unique primitive quadratic character such that $\OQ^{\ker\varepsilon}=\Q(\sqrt{22})$. Computing the first few terms of $f$, we get:
		\begin{equation*}\begin{split}
		f = & \,\, q + \sqrt{-2}q^2 - 2q^3 - 2q^4 + 2\sqrt{-2}q^5 - 2\sqrt{-2}q^6 + q^7-\\
		& - 2\sqrt{-2}q^8 + q^9 - 4q^{10} +(\sqrt{-2} - 3)q^{11} + 4q^{12} - 4q^{13} + \sqrt{-2}q^{14}-\\
		& - 4\sqrt{-2}q^{15} + 4q^{16} + 2\sqrt{-2}q^{17} + \sqrt{-2}q^{18} - 4\sqrt{-2}q^{20} + O(q^{21}).
		\end{split}\end{equation*}
		Using the $q$-expansion above it is easy to see that $a_{616}(f)=4+6\sqrt{-2}$, so that by \cite[Theorem~2.1]{atkli} we get that $\displaystyle\eta_f=\frac{\sqrt{88}}{4+6\sqrt{-2}}$.

		\subsubsection*{Example 2}
		Let $u=-3/4$, $d=-6$, so that $K=\Q(\sqrt{-6})$, and $\displaystyle b=\frac{12}{7}+\frac{2}{7}\sqrt{-6}$. Then a global integral model for the curve $E_{-6,-3/4}^{(7)}$ twisted by $b$ is
		$$E\colon y^2=x^3-(4027482-1132380\sqrt{-6})x +2581493976-1335076020\sqrt{-6},$$
		which has $j$-invariant
		$$j(E)=- \frac{12097712691}{78125}+\frac{10861109532}{78125}\sqrt{-6}\notin \mathcal O_K.$$
		There is an isogeny $\mu\colon E\to \n E$ of degree $7$, whose composition with $\n\mu$ coincides with multiplication by $7$. Setting $\displaystyle \omega_{E}=\frac{dx}{2y}$ we obtain that $\mu^*(\omega_{\n E})=(-1+\sqrt{-6})\omega_{E}$ so $\alpha=-1+\frac{1}{2}\sqrt{-24}$ and $\beta=-2+2\sqrt{7}$. By Lemma \ref{denomid}, $D_h=(1)$. The conductor of $E$ is given by
		$$\mathcal N(E)=(480)=(2)^5(3)(5)=(2, \sqrt{-6})^{10} (3, \sqrt{-6})^2 (5, 2+\sqrt{-6}) (5,3+ \sqrt{-6}).$$
		The Weierstrass ideal attached to the standard invariant differential is
		$$\delta_{\omega_E}=\left(\frac{1}{21}+\frac{1}{21}\sqrt{-6}\right)$$
		and has norm $1/63$.
		The isogeny graph of $E$ is given by:
		$$\begin{xymatrix}{
				   E\ar@{-}[r]^{7}\ar@{-}[d]_{2} & \n E \ar@{-}[d]^{2}\\
					E_1 \ar@{-}[r]^7 & \n E_1.
			  }\end{xymatrix}$$
			\begin{center}
				\begin{tabular}{| c | c | c | c | c | c |}
			    \hline
				$\Omega_E$ & $q$ & $|E(K)_{\text{tors}}|$ & $t\cdot N_{K/\Q}(D_{\alpha})$  & $s^2$ \\ \hline\hline
			    $0.663037499513841$ & $1/2$ & $2$ & $1$  & $4$\\ \hline
			  \end{tabular}
		\end{center}
		Theorem \ref{mainthm} shows that if $L(E/K,1)\neq 0$ then:
		$$|L(E/K,1)|\geq\frac{0.663037499513841}{4\cdot \sqrt{24}\cdot 4}\approx 8.45887267706248\cdot 10^{-3}.$$
		The points 
		$$P=\left(\frac{29502}{25}-\frac{3546}{25} \sqrt{-6} , -\frac{391554}{125}+\frac{59292}{125} \sqrt{-6} \right)$$
		and
		$$Q=\left(- 1674 -1287 \sqrt{-6} : - 252288-7776 \sqrt{-6} \right)$$
		in $E(K)$ are independent and they have infinite order.

		The newform $\displaystyle f=\sum_{n=1}^{+\infty}a_nq^n$ attached to $E$ has level $480\cdot 24=11520$ and since $F=\Q(a_n\colon n\in \N)=\Q(\sqrt{7})$, the character of $f$ is trivial, so $f\in S_2(\Gamma_0(11520))$. The first coefficients of the $q$-expansion of $f$ are:
		$$f=q+q^5+2q^7-2q^{11}+2\sqrt{7}q^{19}+O(q^{21}).$$
		The sign $\eta_f$ of the functional equation for $f$ is $-1$.

		\subsubsection*{Example 3}
		Let $d=34$, $u=7/4$, $b=17/2+3/2\sqrt{34}$. An integral model for $E_{34,7/4}^{(2)}$ twisted by $b$ is given by:
		$$E\colon y^2 = x^3 + (365568+62730\sqrt{34})x -111410656-19106640\sqrt{34}$$
		and the $j$-invariant is
		$$j(E)=\frac{1353090752}{680625}-\frac{123420416}{680625}\sqrt{34},$$
		so that $E$ has no CM.
		There is an isogeny $\mu\colon E\to \n E$ of degree $2$ given by $(x,y)\mapsto (g(x),y\cdot h(x))$ as follows:
		$$g(x)=\frac{(35/2-3\sqrt{34})x^2 + (68-12\sqrt{34})x + 612+ 1071\sqrt{34}}{x - 136-24\sqrt{34}}$$
		$$h(x)=\frac{-(207/2+71/4\sqrt{34})x^2 - (816+140\sqrt{34})x - 18003-3179\sqrt{34} }{x^2 - (272-48\sqrt{34})x + 38080-6528\sqrt{34} }$$
		and $\n\mu\mu$ coincides with multiplication by $2$, so $F=\Q(\sqrt{2})$. Using $\displaystyle \omega_{E}=\frac{dx}{2y}$ we get $\displaystyle \alpha=-6-\frac{1}{2}\sqrt{136}$ and $\beta=12-2\sqrt{2}$. The conductor of $E$ is
		$$\mathcal N(E)=(1077120)=(2)^7(3)^2(5)(11)(17)=(6-\sqrt{34})^{14}(3,1+ \sqrt{34})^2(3,2+ \sqrt{34})^2\cdot $$
		$$\cdot (5,2+ \sqrt{34})(5,3+ \sqrt{34})(11, 10+\sqrt{34})(11, 1+\sqrt{34})(17-3\sqrt{34})^2.$$
		One can check that the given Weierstrass equation is a global minimal model for $E$, so that $\delta_{\omega_{E}}=(1)$. Also, the isogeny graph of $E$ is
		$$\begin{xymatrix}{E\ar@{-}[r]^2 & \n E,}\end{xymatrix}$$
		so that conjecturally $E$ possesses an optimal parametrization; therefore we can set $s=1$.
			\begin{center}
				\begin{tabular}{| c | c | c | c | c |}
			    \hline
				$\Omega_E$ & $q$ & $|E(K)_{\text{tors}}|$ & $t\cdot N_{K/\Q}(D_{\alpha})$  & $s^2$ \\ \hline\hline
			    $0.0704074944313492$ & $-1/2$ & $2$ & $1$ &  $1$\\ \hline
			  \end{tabular}
		\end{center}
		The lower bound given by Theorem \ref{mainthm} is:
		$$|L(E/K,1)|\geq\frac{0.0704074944313492}{4\cdot \sqrt{136}}\approx 1.50934820976064\cdot 10^{-3}.$$
		Two independent points of infinite order in $E(K)$ are given by:
		$$P=(1768+300 \sqrt{34} , 107100+18360 \sqrt{34})$$
		and
		$$Q=\left(\frac{867}{4} +\frac{65}{2} \sqrt{34} , - \frac{48025}{8}-\frac{8075}{8} \sqrt{34}\right).$$
		The newform $\displaystyle f=\sum_{n=1}^{+\infty}a_nq^n$ attached to $E$ has level $1077120\cdot 136=146488320$ and since $F=\Q(a_n\colon n\in \N)=\Q(\sqrt{2})$, the character of $f$ is trivial, so $f\in S_2(\Gamma_0(146488320))$. Its $q$-expansion is:
		$$f=q+q^5-q^{11}+4\sqrt{2}q^{13}+5\sqrt{2}q^{19}+O(q^{21}),$$
		and the sign of the functional equation is again $-1$. 

		The next example, borrowed from \cite[Proposition 10]{ell}, exhibits a curve of algebraic rank $2$ whose field of definition $K$ coincides with the field $F$ generated by the Fourier coefficients of the attached newform.

		\subsubsection*{Example 4}
		Let $K=\Q(\sqrt{2})$ and $E\colon y^2=x^3+(8+8\sqrt{2})x^2+(16+10\sqrt{2})x$. The $j$-invariant of $E$ is $\displaystyle \frac{698048}{49}+\frac{379136}{49}\sqrt{2}$. An isogeny $\mu\colon E\to \n E$ of degree $2$ is given by $(x,y)\mapsto (g(x),y\cdot h(x))$, where
		$$g(x)=\frac{(3/2-\sqrt{2})x^2 - (4-4\sqrt{2})x +4-\sqrt{2}}{x}$$
		$$h(x)=\frac{(- 5/2+7/4\sqrt{2})x^2 + 5-3\sqrt{2}}{x^2}.$$
		The isogeny $\n\mu\mu$ coincides with multiplication by $2$, so $F=K=\Q(\sqrt{2})$. The standard invariant differential $\displaystyle \omega_E=\frac{dx}{2y}$ gives us $\alpha=-2-\frac{1}{2}\sqrt{8}$ and $\beta=4-2\sqrt{2}$. The conductor of $E$ is
		$$\mathcal N(E)=(896)=(2)^7(7)=(\sqrt{2})^{14}(1-2\sqrt{2})(1+2\sqrt{2}).$$
		The isogeny graph of $E$ is given by
		$$\begin{xymatrix}{E\ar@{-}[r]^2 & \n E}\end{xymatrix}.$$
		The given Weierstrass equation is a global minimal model for $E$.
			\begin{center}
				\begin{tabular}{| c | c | c | c | c |}
			    \hline
				$\Omega_E$ & $q$ & $|E(K)_{\text{tors}}|$ & $t\cdot N_{K/\Q}(D_{\alpha})$ & $s^2$ \\ \hline\hline
			    $2.60444072643674$ & $-1/2$ & $2$ & $1$ & $1$\\ \hline
			  \end{tabular}
		\end{center}
		The lower bound given by Theorem \ref{mainthm} is:
		$$|L(E/K,1)|\geq\frac{2.60444072643674}{4\cdot 8}\approx 8.13887727011480\cdot 10^{-2}.$$
		Two independent points of finite order in $E(K)$ are 
		$$P=(-2 \sqrt{2} ,-4 -2 \sqrt{2})\text{ and }Q=(1-2 \sqrt{2}, 1-2 \sqrt{2}).$$
		The newform $\displaystyle f=\sum_{n=1}^{+\infty}a_nq^n$ attached to $E$ has level $896\cdot 8=7168$ and has trivial character, so $f\in S_2(\Gamma_0(7168))$. The first terms of its $q$-expansion are:
		$$f=q-\sqrt{2}q^3-2\sqrt{2}q^5-q^7-q^9-4\sqrt{2}q^{11}+4\sqrt{2}q^{13}+4q^{15}-2q^{17}+\sqrt{2}q^{19}+O(q^{21}),$$
		and the sign of the functional equation is $-1$.

		The next one is our last example of a curve of positive algebraic rank, which is the curve $E_{109,865}^{(-3)}$ in \cite[Table~2]{has}. This is a curve whose $L$-function cannot feasibly be treated with modular symbols methods, due to the size of the level of the associated newform. Our algorithm furnishes a fast way to verify that its analytic rank is positive.
		\subsubsection*{Example 5}
		Let $K=\Q(\sqrt{109})$ and let
		$$E\colon y^2 + \left(\frac{1+\sqrt{109}}{2}\right) x y +\left(\frac{1+\sqrt{109}}{2}\right) y = $$
		$$=x^{3} + \left(\frac{3-\sqrt{109}}{2}\right) x^{2} + \left(- 223070-21370 \sqrt{109} \right) x -\frac{2727437331+261241129\sqrt{109}}{2}.$$
		This is a $\Q$-curve, given in a global minimal model, with no CM, since its $j$-invariant is not integral. There is an isogeny $\mu\colon E\to \n E$ of degree $3$, which coincides with multiplication by $-3$ when composed with $\n\mu$. Setting $\displaystyle \omega_E=\frac{dx}{2y}$ we obtain $\displaystyle \alpha=-\frac{73}{2}-\frac{7}{2}\sqrt{109}$ and thus $\displaystyle\beta=\frac{73}{7}-\frac{2}{7}\sqrt{-3}$. The conductor of $E$ is given by:
		$$\mathcal N(E)=(755153)=\left(-\frac{9}{2}+\frac{1}{2} \sqrt{109} \right)\left(\frac{9}{2}+\frac{1}{2} \sqrt{109} \right)\left(\frac{7}{2}-\frac{3}{2} \sqrt{109}\right) \left(\frac{7}{2}+\frac{3}{2} \sqrt{109} \right)\cdot$$
		$$\cdot\left(-38-3 \sqrt{109}\right)\left(-38+3 \sqrt{109}\right).$$
		The isogeny graph of $E$ is given by:
		$$\begin{xymatrix}{E\ar@{-}[r]^3 & \n E}\end{xymatrix}.$$
			\begin{center}
				\begin{tabular}{| c | c | c | c | c |}
			    \hline
				$\Omega_E$ & $q$ & $|E(K)_{\text{tors}}|$ & $t\cdot N_{K/\Q}(D_{\alpha})$ & $s^2$ \\ \hline\hline
			    $0.294164545914390$ & $-7/2$ & $1$ & $4$ & $1$\\ \hline
			  \end{tabular}
		\end{center}
		The lower bound for $|L(E/K,1)|$ is:
		$$|L(E/K,1)|\geq \frac{0.294164545914390}{109\cdot 7\cdot 4}\approx 9.63841893559600\cdot 10^{-5}.$$
		A point of infinite order is given by:
		$$P=\left(\frac{119855}{98}+\frac{18381}{196} \sqrt{109}, \frac{28054277}{686}+\frac{9570003}{2744} \sqrt{109}\right).$$
		It is possible to check that $Q\coloneqq \n\mu(\n P)$ is a $K$-rational point of $E$ which is linearly independent with $P$.
		
		The level of the newform $f$ attached to $E$ is $755153\cdot 109=82311677$ and its character is the quadratic character attached to $\Q(\sqrt{109})$ by class field theory. The $q$-expansion of $f$ is given by:
		\begin{equation*}\begin{split}
		f= & \,\, q-\sqrt{-3}q^2-q^3-q^4+\sqrt{-3}q^6+q^7-\sqrt{-3}q^8-2q^9+2\sqrt{-3}q^{11}+q^{12}+\\
		& +2\sqrt{-3}q^{13}-\sqrt{-3}q^{14}-5q^{16}+2\sqrt{-3}q^{18}+\sqrt{-3}q^{19}+O(q^{21}).
		\end{split}\end{equation*}
		By computing the $q$-expansion to a greater precision, one can verify that, by \cite[Theorem~2.1]{atkli}, the sign of the functional equation for $f$ is $\displaystyle -\frac{\sqrt{109}}{1+6\sqrt{-3}}$.

		We used T.\ Dokchitser's PARI/GP script ``computeL'' \cite{computeL}, also included in Sage, to check that for every newform $f$ computed above, at least the first $14$ significant digits of $L(f,1)$ and of $L(\s f,1)$ are equal to $0$, while $L'(f,1),L'(\s f,1)\neq 0$. Since $L(E/K,s)=L(f,s)\cdot L(\s f,s)$, this proves that $L(E/K,1)=L'(E/K,1)=0$, while $L''(E/K,1)\neq 0$. Therefore, assuming Conjecture~\ref{genmancon}, all curves in the above examples have analytic rank $2$.

		A rigorous analysis of the error in floating-point computations, even if possible in principle, is beyond the goal of the present work. However, we repeated the computations several times using different high precisions, and it is very unlikely that the floating-point error has any significant influence on the outcome.

		We conclude with one last example of a $\Q$-curve $E$, coming from the Hecke involution on $X_1(13)$, such that $m=-1$. This means that $E$ is isomorphic to $\n E$ and therefore both curves are isomorphic over $\OQ$ to an elliptic curve defined over $\Q$, but no isomorphism $E\to \n E$ descends to $\Q$. Using for example the algorithm described in \cite{sim}, it is possible to check that the curve given in this example has algebraic rank $0$, and we will use our main theorem to compute its $L$-ratio.

		\subsubsection*{Example 6}
		Let $a$ be a root of the polynomial $x^2+x-4$ and let $K\coloneqq \Q(a)=\Q(\sqrt{17})$. Then the elliptic curve
		$$y^2 + (1373+536 a) x y + (482701840+188441104 a) y = x^3 + (244408+95414 a) x^2$$
		is a $\Q$-curve with $j$-invariant $\displaystyle -\frac{60698457}{40960}$ (this is curve 2.2.17.1-100.1-e2 in the LMFDB database \cite{lmfdb}). There is an isomorphism $\mu\colon E\to \n E$ given by $(x,y)\mapsto (g(x,y),h(x,y))$, where
		$$g(x,y)=(473754361-303386704 a) x-214320+ 137248a$$
		\begin{equation*}\begin{split}
		h(x,y)& =(-587942141286+376511211445 a) x\\
			 & -(16755744253243-10730180955650 a ) y\\
			& + 100734048- 64508896 a.
		\end{split}\end{equation*}
		Using $\displaystyle\omega_E=\frac{dx}{2y+(1373+536 a)y+482701840+188441104 a}$, we get $\displaystyle\alpha=17684+4289\sqrt{17}$, which gives immediately $D_h=(1)$. The conductor of $E$ is given by
		$$\mathcal N(E)=(10)=(1-a)(2+a)(5).$$
		Note that $E$ has non-split multiplicative reduction at $5$, which is an inert prime in $K$. The isogeny graph of $E$ is simply
		$$\begin{xymatrix}{E\ar@{-}[r]^{\sim}\ar@{-}[d]_{13} & \n E\ar@{-}[d]^{13} \\ E'\ar@{-}[r]^{\sim} & \n E'.}\end{xymatrix}$$
		The given Weierstrass equation for $E$ is a global minimal model.

		Here we have the table of the invariants of $E$ used in Theorem \ref{mainthm}:
			\begin{center}
				\begin{tabular}{| c | c | c | c | c |}
			    \hline
				$\Omega_E$ & $q$ & $|E(K)_{\text{tors}}|$ & $t\cdot N_{K/\Q}(D_{\alpha})$  & $s^2$ \\ \hline\hline
			    $11.1808314690274$ & $4289$ & $13$ & $1$ & $169$\\ \hline
			  \end{tabular}
		\end{center}
		By Lemma \ref{Lratio}, in order to compute the $L$-ratio of $E$, i.e. the value $\displaystyle L(E,1)\cdot \frac{\sqrt{17}}{\Omega_E}\in \Q^*$, we only need a bound on the denominator of such number. By Theorem \ref{mainthm}, this is given by:
		$$B=2\cdot 4289\cdot 169\cdot 169=244996258.$$
		The $L$-ratio for $E$ is given by:
		$$L(E,1)\cdot\frac{\sqrt{17}}{\Omega_E}=1.$$
		
		Since the Tamagawa numbers of $E$ at the prime ideals $(1-a),(2+a)$ and $(5)$ are respectively $13,13$ and $1$, the strong BSD conjecture would imply that
		$$L(E,1)\cdot\frac{\sqrt{17}}{\Omega_E}=|\Sh(E/K)|=1.$$
		It is possible to check using the algorithm given in \cite{sim} that $\Sh(E/K)[2]$ is trivial.

		The newform $f$ attached to $E$ belongs to $S_2(\Gamma_1(170),\varepsilon)$, for $\varepsilon$ the unique primitive quadratic character such that $\OQ^{\ker\varepsilon}=K$.
		\begin{equation*}\begin{split}
		f= &\,\, q + q^{2} + 3iq^{3} + q^{4} + iq^{5} + 3iq^{6} -4iq^{7} + q^{8} - 6q^{9} + iq^{10} + 2iq^{11} + 3iq^{12}\\
		 & + q^{13} -4iq^{14} - 3q^{15} + q^{16} -(4+i)q^{17} - 6q^{18} + 7q^{19} + iq^{20} + O(q^{21}).
		\end{split}
		\end{equation*}
		Using \cite[Theorem 2.1]{atkli} we get that the sign of the functional equation for $f$ is $\displaystyle \eta_f=\frac{\sqrt{17}}{1+4i}$. It is possible to check numerically that $L(E,1)\neq 0$.
		
\bibliographystyle{amsplain}
\bibliography{biblio}

\end{document}